\documentclass[aop]{imsart}

\RequirePackage{amsthm,amsmath,amsfonts,amssymb}
\RequirePackage[numbers]{natbib}

\RequirePackage[colorlinks,citecolor=blue,urlcolor=blue]{hyperref}
\usepackage{cleveref}
\usepackage[dvipdfmx]{graphicx}
\usepackage{epstopdf}
\usepackage{amsopn}
\usepackage{bbm}

\DeclareMathOperator*{\argmax}{arg\,max}
\DeclareMathOperator*{\argmin}{arg\,min}
\newcommand{\R}{\mathbb{R}}
\newcommand{\wh}{\widehat}
\newcommand{\bs}{\boldsymbol}
\newcommand{\1}{\mathbbm{1}}
\newcommand{\V}{\mathcal{V}}
\newcommand{\C}{\mathcal{C}}
\newcommand{\E}{\mathbb{E}}
\newcommand{\Var}{\mathrm{Var}}
\newcommand{\relmiddle}[1]{\mathrel{}\middle#1\mathrel{}}

\startlocaldefs

\newtheorem{theorem}{Theorem}[section]
\newtheorem{proposition}[theorem]{Proposition}
\newtheorem{lemma}[theorem]{Lemma}
\newtheorem{corollary}[theorem]{Corollary}

\theoremstyle{remark}
\newtheorem{definition}[theorem]{Definition}
\newtheorem{remark}[theorem]{Remark}

\crefname{equation}{}{}
\crefname{theorem}{Theorem}{Theorems}
\crefname{proposition}{Proposition}{Propositions}
\crefname{lemma}{Lemma}{Lemmas}
\crefname{corollary}{Corollary}{Corollarys}
\crefname{definition}{Definition}{Definitions}
\crefname{remark}{Remark}{Remarks}

\endlocaldefs
\begin{document}
\begin{frontmatter}
\title{Decision tree-based estimation of the overlap of two probability distributions}
\runtitle{Estimation of distribution overlap}

\begin{aug}
\author[A]{\fnms{Hisashi} \snm{Johno}\ead[label=e1]{johnoh@yamanashi.ac.jp}}
\and
\author[B]{\fnms{Kazunori} \snm{Nakamoto}\ead[label=e2]{nakamoto@yamanashi.ac.jp}}
\address[A]{Department of Radiology, Faculty of Medicine, University of Yamanashi, Japan, \printead{e1}}

\address[B]{Center for Medical Education and Sciences, Faculty of Medicine, University of Yamanashi, Japan, \printead{e2}}
\end{aug}

\begin{abstract}
A new nonparametric approach, based on a decision tree algorithm, is proposed to calculate the overlap between two probability distributions.
The devised framework is described analytically and numerically.
The convergence of the estimated overlap to the true value is proved along with some experimental results.
\end{abstract}



\end{frontmatter}


\section{Introduction}
\label{sec:intro}
In various scientific fields, it is important to assess the similarity between data sets or distributions.
The overlap coefficient (OVL) is an interpretable measure of such similarity, defined as the common area under two probability density functions (PDFs).
While a variety of parametric techniques to estimate OVL have been developed, existing nonparametric ones are wholly based on kernel density estimation (KDE) \cite{montoya19, pastore19, schmid06}.
Although KDE is a useful and widely practiced method to estimate probability density functions, the optimal setting of its parameters (kernel function and bandwidth) is a challenging task.

Here we propose a new nonparametric method to calculate OVL based on a decision tree algorithm.
We start with notation and preliminaries in \Cref{sec:prelim}.
The devised framework is described analytically in \Cref{sec:anal} and numerically in \Cref{sec:numer}.
Experimental results are shown in \Cref{sec:exper}, and the conclusion follows in \Cref{sec:conclusion}.

\section{Preliminaries}
\label{sec:prelim}
Let $f_1$ and $f_2$ be two continuous PDFs on the real line $\R$.
The OVL between $f_1$ and $f_2$ is defined as
\begin{equation*}
\rho(f_1, f_2) = \int_{-\infty}^\infty \min\left\{f_1(x), f_2(x)\right\}\:dx.
\end{equation*}

\begin{definition}
\upshape
Suppose $g_1$ and $g_2$ are real continuous functions on $\R$.
Then we call $x\in\R$ a {\em crossover point} between $g_1$ and $g_2$ if there exist points $a, b$ in any neighborhood of $x$ such that $[g_1(a)-g_2(a)][g_1(b)-g_2(b)]<0$.
We also call $x\in\R$ a {\em coincidence point} between $g_1$ and $g_2$ if $g_1(x)=g_2(x)$.
The set of crossover points and that of coincidence points are denoted by $C(g_1, g_2)$ and $C'(g_1, g_2)$, respectively.
Note that $C(g_1, g_2)\subset C'(g_1, g_2)$.
\end{definition}

Under the assumption that $C'(f_1, f_2)$ is finite and the cardinality of $C(f_1, f_2)$ is known in advance, we present a decision tree-based method to estimate $\rho(f_1, f_2)$.
The rest of this section provides further notations and terminologies.
\begin{definition}\label{prob_space}
\upshape
Let $(\Omega, \mathcal{F}, \mathbb{P})$ be a probability space and $(X, Y):\Omega\to\R\times\{1,2\}$ a random variable with distribution $P$, defined as $P(A) = \mathbb{P}((X, Y)^{-1}(A))$ for all Borel sets $A\subset\R\times\{1,2\}$.
From the viewpoint of binary classification, the measurable functions $X:\Omega\to\R$ and $Y:\Omega\to\{1, 2\}$ can be regarded as explanatory and response variables, respectively.
Given a Borel set $B\subset\R$, we may simply write $P(X\in B)$ for $P(B\times\{1, 2\})$, $\pi_j$ for $P(\R\times\{j\})$, $F_j (x)$ for $P((-\infty, x]\times\{j\})/\pi_j$, $P(X\in B, Y=j)$ for $P(B\times\{j\})$, and $P(Y=j\mid X\in B)$ for $P(X\in B, Y=j)/P(X\in B)$, provided $\pi_j\ne 0$ and $P(X\in B)\ne 0$ as necessary.
\end{definition}

We shall consider the random variable $(X, Y)$ with
\begin{equation*}
F_j(x) = \int_{-\infty}^x f_j(t)\: dt \qquad (x\in\R;\ j=1, 2),
\end{equation*}
so that each $F_j$ is the cumulative distribution function (CDF) corresponding to the continuous PDF $f_j$.
We also define $F_j(-\infty)=0$ and $F_j(\infty)=1$ ($j=1,2$).



\begin{definition}
\upshape
Let $\Delta^1$ be the standard $1$-simplex,  which consists of all points $(a,b)\in\R^2$ such that $a+b=1$, $a\ge 0$, and $b\ge 0$.
An {\em impurity function} on $\Delta^1$ is a function $\iota$ with the following properties:
\begin{enumerate}
\item $\iota$ attains its maximum only at $(1/2, 1/2)$,
\item $\iota$ attains its minimum only at $(1, 0)$ and $(0, 1)$,
\item $\iota$ is a symmetric function, i.e., $\iota(a, b) = \iota(b, a)$.
\end{enumerate}
\end{definition}

\begin{definition}\label{V}
\upshape
For a positive integer $m$, let $\R_\le^m$ be the set of all $\bs{v}=(v_1,\ldots,v_m)\in\R^m$ with $v_1\le\cdots\le v_m$.
By the {\em $(m+1)$-ary split} on $\R$ at a point $\bs{v}\in\R_\le^m$, we mean the collection $S_{\bs{v}} = \{S_{\bs{v},1},\ldots,S_{\bs{v},m+1}\}$ with $S_{\bs{v},1}=\{x\in\R\mid x\le v_1\}$, $S_{\bs{v},m+1}=\{x\in\R\mid x>v_m\}$, and $S_{\bs{v},k}=\{x\in\R\mid v_{k-1}<x\le v_k\}$ for $k=2,\ldots,m$.
Note that each $S_{\bs{v},k}$ is a Borel set in $\R$, $S_{\bs{v},k}\cap S_{\bs{v},l} = \emptyset$ if $k\ne l$, and $S_{\bs{v},1}\cup\cdots\cup S_{\bs{v},m+1}=\R$.
\end{definition}

Using an impurity function $\iota$ on $\Delta^1$, we define the {\em impurity} of a Borel set $B\subset\R$ for the binary classification by
\begin{equation*}
I(B) = \begin{cases}
\ \iota\left(P(Y=1\mid X\in B), P(Y=2\mid X\in B)\right) & \mbox{if}\quad P(X\in B)>0,\\
\ 0 & \mbox{if}\quad P(X\in B)=0,
\end{cases}
\end{equation*}
and the {\em goodness} of $S_{\bs{v}}$ ($\bs{v}\in\R_\le^m$) by
\begin{equation}\label{delI}
\varDelta I(S_{\bs{v}}) = I(\R)-\sum_{k=1}^{m+1} P(X\in S_{\bs{v},k}) I(S_{\bs{v},k}),
\end{equation}
according to the conventional decision tree algorithm \cite{breiman84}.
If there exists $\bs{v}'\in\R_\le^m$ such that $\varDelta I(S_{\bs{v}'}) = \sup\varDelta I(S_{\bs v})$, where the supremum is over all ${\bs v}\in\R_\le^m$, then we call $S_{\bs{v}'}$ a {\em best $(m+1)$-ary split} on $\R$.

\section{Analytical framework}
\label{sec:anal}
In this section, we present the theoretical foundation of our method to calculate $C(\pi_1f_1, \pi_2f_2)$ and $\rho(\pi_1f_1, \pi_2f_2)$ under the assumptions that $C'(\pi_1f_1, \pi_2f_2)$ is finite, the cardinality $n$ of $C(\pi_1f_1, \pi_2f_2)$ is known in advance, $\pi_1>0$, and $\pi_2>0$.
We can obtain $C(f_1, f_2) = C(\pi_1f_1, \pi_2f_2)$ and $\rho(f_1, f_2) = 2\rho(\pi_1f_1, \pi_2f_2)$ if $\pi_1 = \pi_2 = 1/2$, which may be realized with sampling techniques, e.g., drawing the same number of samples from both the distributions corresponding to $f_1$ and $f_2$.
Here we use the setting of the previous section and, in addition, adopt the misclassification-based impurity function \cite{breiman84}, i.e., 
\begin{equation}\label{iota}
\iota(a, b) = 1-\max\left\{a, b\right\}\qquad ((a, b)\in\Delta^1).
\end{equation}

Suppose $C(\pi_1f_1, \pi_2f_2)=\emptyset$, or $n=0$.
Then either $\pi_1f_1\le\pi_2f_2$ or $\pi_1f_1\ge\pi_2f_2$ holds.
(Recall that $C'(\pi_1f_1, \pi_2f_2)$ is finite.)
In the former case, we have $\rho(\pi_1f_1, \pi_2f_2) = \pi_1$, and in the latter, $\rho(\pi_1f_1, \pi_2f_2) = \pi_2$.
Of note, $\pi_1 = \pi_2 = 1/2$ cannot occur here.

In the following, we assume $C(\pi_1f_1, \pi_2f_2)\ne\emptyset$, so that $n$ is a positive integer.
Put $C(\pi_1f_1, \pi_2f_2) = \{c_1,\ldots,c_n\}$ with $c_1<\cdots<c_n$, $\bs{c}=(c_1,\ldots,c_n)$, $c_0=-\infty$, and $c_{n+1}=\infty$.
The $(n+1)$-ary split on $\R$ at $\bs{c}$ is defined by $S_{\bs{c}} = \{S_{\bs{c},1},\ldots,S_{\bs{c},n+1}\}$ (see \Cref{V}).
\Cref{fig:def} is a schematic example to illustrate $C(\pi_1f_1, \pi_2f_2)$ and $\rho(\pi_1f_1, \pi_2f_2)$.

\begin{figure}[htpb]
\centering
\includegraphics[width=0.63\textwidth]{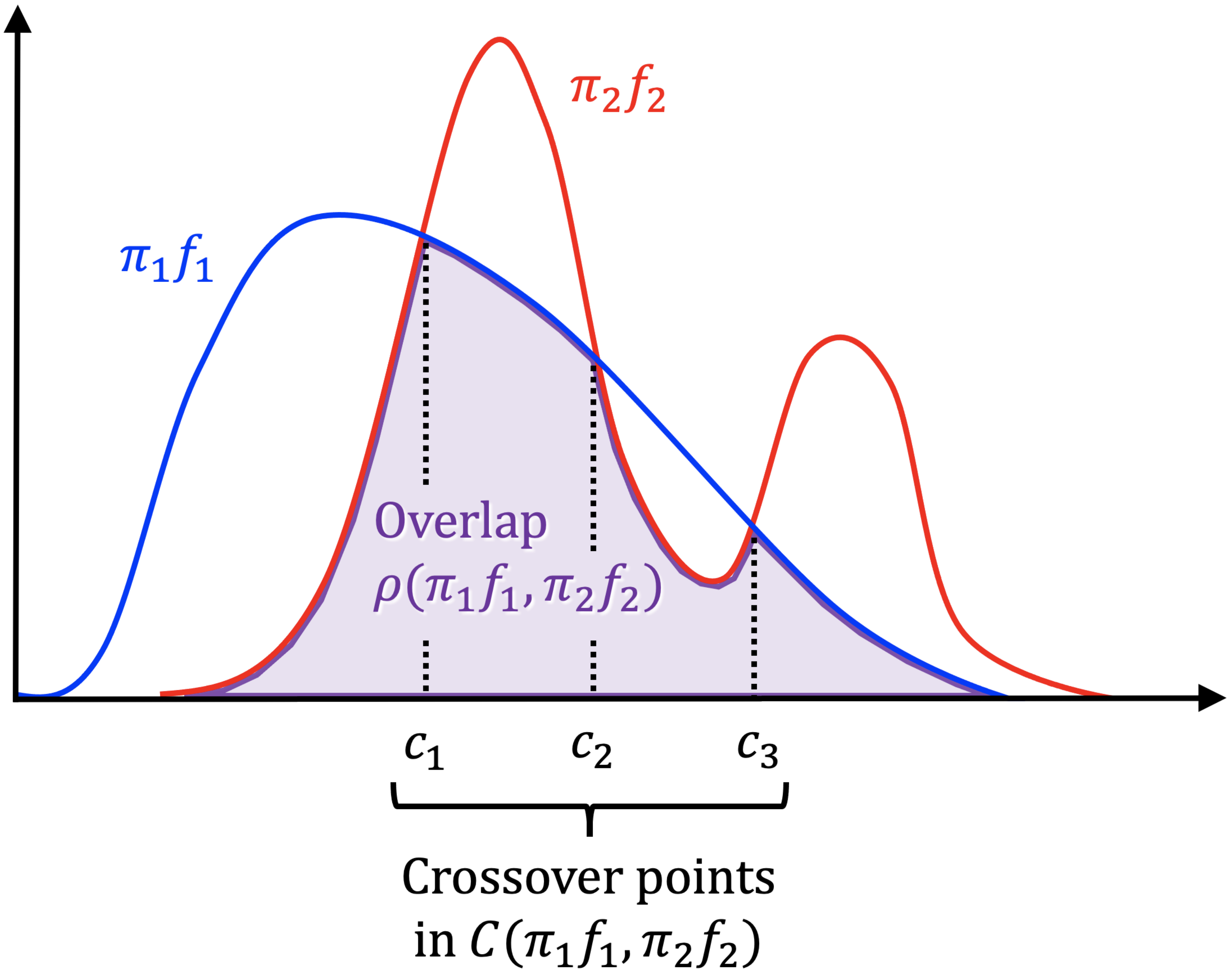}
\caption{A schematic example of $C(\pi_1f_1, \pi_2f_2)$ and $\rho(\pi_1f_1, \pi_2f_2)$.}
\label{fig:def}
\end{figure}

\begin{proposition}\label{delI_F}
For $\bs{v}=(v_1,\ldots,v_m)\in\R_\le^m$ with $m$ a positive integer, we have
\begin{align*}
\varDelta I(S_{\bs v})
&= \sum_{k=1}^{m+1}\max_j\left\{\pi_j\left[F_j(v_k)-F_j(v_{k-1})\right] \right\} - \max_j \left\{\pi_j\right\} \\
&= \sum_{k=1}^{m+1} \max_j\left\{\int_{v_{k-1}}^{v_k}\pi_jf_j(x)\:dx \right\}
- \max_j \left\{\pi_j\right\},
\end{align*}
where $v_0=-\infty$ and $v_{m+1}=\infty$.
\end{proposition}
\begin{proof}
From \Cref{delI,iota}, we have
\begin{equation}\label{delI2}
\varDelta I(S_{\bs v}) = \sum_k P(X\in S_{\bs{v},k})\max_j\left\{P(Y=j\mid X\in S_{\bs{v},k})\right\} - \max_j\left\{P(Y=j\mid X\in\R)\right\},
\end{equation}
where the sum is over all $k$ with $P(X\in S_{\bs{v},k})>0$.
Since $P(Y=j\mid X\in S_{\bs{v},k}) = P(S_{\bs{v},k}\times\{j\})/P(X\in S_{\bs{v},k})$ and $P(S_{\bs{v},k}\times\{j\}) = \pi_j[F_j(v_k)-F_j(v_{k-1})]$,
we obtain
\begin{equation*}
P(X\in S_{\bs{v},k})\max_j\left\{P(Y=j\mid X\in S_{\bs{v},k})\right\} = \max_j\left\{\pi_j\left[F_j(v_k)-F_j(v_{k-1})\right] \right\}.
\end{equation*}
As for the last term of \Cref{delI2}, we have $P(Y=j\mid X\in\R)=\pi_j$ by definition.
\end{proof}

The following corollary is immediate from \Cref{delI_F}.

\begin{corollary}\label{g_cor}
For $\bs{v}=(v_1,\ldots,v_m)\in\R_\le^m$ with $m$ a positive integer, let 
\begin{equation*}
g_{\bs v}(x) = \pi_{j_k}f_{j_k}(x) \qquad (x\in S_{\bs{v},k};\ k=1,\ldots,m+1)
\end{equation*}
where $j_k\in\argmax_j\,\{\pi_j[F_j(v_k)-F_j(v_{k-1})]\}$.
Let $g=\max\,\{\pi_1f_1, \pi_2f_2\}$.
Then 
\begin{equation*}
\varDelta I(S_{\bs v}) =\int_{-\infty}^\infty g_{\bs v}(x)\:dx - \max_j \left\{\pi_j\right\}, \quad
\varDelta I(S_{\bs c}) =\int_{-\infty}^\infty g(x)\:dx - \max_j \left\{\pi_j\right\}.
\end{equation*}
Furthermore, $g_{\bs v}\le g$ and $\varDelta I(S_{\bs v}) \le \varDelta I(S_{\bs c})$.
\end{corollary}

\begin{lemma}\label{delI_lem}
Suppose $m$ is a positive integer, $\bs{v}=(v_1,\ldots,v_m)\in\R_\le^m$, $v_0=-\infty$, and $v_{m+1}=\infty$.
If $v_{k-1}<c_p<v_k$ for some $k\in\{1,\ldots,m+1\}$ and $p\in\{1,\ldots,n\}$, then $\varDelta I(S_{\bs v}) < \varDelta I(S_{\bs c})$.
\end{lemma}

\begin{proof}
Since $C'(\pi_1f_1, \pi_2f_2)$ is finite, there exists a neighborhood $U$ of $c_p$ such that $U\subset (v_{k-1}, v_k)$ and $U\cap C'(\pi_1f_1, \pi_2f_2) = \{c_p\}$.
Then, $[\pi_1f_1(a)-\pi_2f_2(a)][\pi_1f_1(b)-\pi_2f_2(b)]<0$ for all $a,b\in U$ with $a<c_p<b$.
Without loss of generality, we assume that $\pi_1f_1(a)<\pi_2f_2(a)$ and $\pi_2f_2(b)<\pi_1f_1(b)$.
If $g_{\bs v} = \pi_1f_1$ on $S_{\bs{v},k}$, then $g_{\bs v}<g$ on the open interval $(a, c_p)$, so that
\begin{equation*}
\varDelta I(S_{\bs c}) - \varDelta I(S_{\bs v})
\ge \int_a^{c_p}\left[g(x)-g_{\bs v}(x) \right]\:dx > 0.
\end{equation*}
The proof for the case $g_{\bs v} = \pi_2f_2$ on $S_{\bs{v},k}$ is similar.
\end{proof}

\begin{proposition}\label{delISc}
The supremum of $\varDelta I(S_{\bs v})$ over $\bs{v}\in\R_\le^n$ is uniquely attained at $\bs{v} = \bs{c}$.
\end{proposition}

In other words, $S_{\bs c}$ is the unique best $(n+1)$-ary split on $\R$.

\begin{proof}
If $\bs{v}\ne\bs{c}$, then $c_p\notin\{v_1,\dots,v_n\}$ for some $p$.
Hence $v_{k-1}<c_p<v_k$ for some $k$ as in the assumption of \Cref{delI_lem}, so that $\varDelta I(S_{\bs v}) < \varDelta I(S_{\bs c})$.
\end{proof}

\begin{proposition}\label{delISc2}
Suppose $m$ is a positive integer with $m<n$. 
Then for every $\bs{v}\in\R_\le^m$, $\varDelta I(S_{\bs v})<\varDelta I(S_{\bs c})$.
\end{proposition}

\begin{proof}
Since $m<n$, $c_p\notin\{v_1,\dots,v_m\}$ for some $p$.
The proof is similar as above.
\end{proof}

Now we see that $C(\pi_1f_1, \pi_2f_2)$ can be obtained by finding $\bs{v}\in\R_\le^n$ that yields the maximum of $\varDelta I(S_{\bs v})$.
Given $C(\pi_1f_1, \pi_2f_2)$, we have
\begin{equation}\label{rhoeq}
\rho(\pi_1f_1, \pi_2f_2) = \sum_{k=1}^{n+1}\int_{c_{k-1}}^{c_k} \min_j\left\{\pi_jf_j(x)\right\}\:dx
= \sum_{k=1}^{n+1}\min_j\left\{P(X\in S_{\bs{c},k}, Y=j)\right\}.
\end{equation}

\section{Numerical framework}
\label{sec:numer}
Here we show how to estimate $C(\pi_1f_1, \pi_2f_2)$ and $\rho(\pi_1f_1, \pi_2f_2)$, given independent and identically distributed (i.i.d.) random variables $(X_1, Y_1), \ldots, (X_N, Y_N)$ with the distribution $P$ on $\R\times\{1, 2\}$.
Let us keep the setting of the previous section.

\begin{definition}
\upshape
For a Borel set $B\subset \R$ and $j\in\{1, 2\}$, put
\begin{align*}
N_X(B) = \#\{i\mid X_i\in B\},\quad N_Y(j) = \#\{i\mid Y_i=j\} \\
N_{XY}(B, j) = \#\{i\mid X_i\in B, Y_i=j\},\quad \wh{\pi}_{j,N} = N_Y(j)/N\\
\wh{P}_N(X\in B) = N_X(B)/N,\quad \wh{P}_N(X\in B, Y=j) = N_{XY}(B, j)/N, \\
\wh{P}_N(Y=j\mid X\in B) = N_{XY}(B, j)/N_X(B) \quad \mbox{if}\quad \wh{P}_N(X\in B)>0,
\end{align*}
where $\#$ denotes the cardinality of a set.
Define
\begin{equation*}
\wh{I}_N(B) = \begin{cases}
\ \iota\left(\wh{P}_N(Y=1\mid X\in B), \wh{P}_N(Y=2\mid X\in B)\right) & \mbox{if}\quad \wh{P}_N(X\in B)>0,\\
\ 0 & \mbox{if}\quad \wh{P}_N(X\in B)=0
\end{cases}
\end{equation*}
and 
\begin{equation}\label{delIh}
\varDelta \wh{I}_N(S_{\bs{v}}) = \wh{I}_N(\R)-\sum_{k=1}^{m+1} \wh{P}_N(X\in S_{\bs{v},k}) \wh{I}_N(S_{\bs{v},k}) \qquad ({\bs v}\in \R_\le^m;\ m=1,2,\ldots)
\end{equation}
as the estimators of $I(B)$ and $\varDelta I(S_{\bs{v}})$, respectively.
\end{definition}

\begin{definition}\label{hR}
\upshape
Let $X_{N:1}\le\cdots\le X_{N:N}$ be the order statistics of $X_1,\ldots,X_N$,
\begin{align*}
&Z_i = (X_{N:i}+X_{N:i+1})/2 \qquad (i=1,\ldots,N-1), \\
&\wh{\R}_N^m = \left\{(Z_{i_1}, \ldots, Z_{i_m})\mid 1\le i_1\le\cdots\le i_m\le N-1\right\}\qquad (m=1,2,\ldots).
\end{align*}
To avoid trivialities, we set $Z_1=X_1$ if $N=1$.
Note that $\wh{\R}_N^m\subset\R_\le^m$ $(m=1,2,\ldots)$ and recall that $n=\#C(\pi_1f_1, \pi_2f_2)$.
Define $\wh{\bs{v}}_N \in \argmax_{\bs{v}\in\wh{\R}_N^n}\{\varDelta \wh{I}_N(S_{\bs{v}})\}$ and
\begin{equation}\label{rhoheq}
\wh{\rho}_{\bs{v},N} = \sum_{k=1}^{m+1}\min_j\left\{\wh{P}_N(X\in S_{\bs{v},k}, Y=j)\right\}\qquad ({\bs v}\in \R_\le^m;\ m=1,2,\ldots).
\end{equation}
We propose $\wh{\rho}_{\wh{\bs{v}}_N,N}$ as an estimator of $\rho(\pi_1f_1, \pi_2f_2)$.
\end{definition}

\begin{definition}\label{ascmp}
\upshape
Let $\xi$ be a random variable and $\{\xi_i\}$ a sequence of random variables on $(\Omega,\mathcal{F},\mathbb{P})$ taking values in a separable metric space $(A, d)$.
We say that $\{\xi_i\}$ {\it converges almost surely} to $\xi$ if
\begin{equation*}
\mathbb{P}\left(\left\{\omega\in\Omega\relmiddle| \lim_{i\to\infty}\xi_i(\omega)=\xi(\omega)\right\}\right) = 1.
\end{equation*}
We also say that $\{\xi_i\}$ {\it converges completely} to $\xi$ if
\begin{equation*}
\sum_{i=1}^\infty\mathbb{P}\left(\left\{\omega\in\Omega\relmiddle| d\left(\xi_i(\omega),\xi(\omega)\right)>\epsilon\right\}\right) < \infty
\end{equation*}
for any $\epsilon>0$.
\end{definition}

\begin{remark}\label{rem:compas}
{\normalfont (See \cite{hsu47} for reference.)}
In \Cref{ascmp}, $\{\xi_i\}$ converges almost surely to $\xi$ if and only if
\begin{equation*}
\lim_{l\to\infty}\mathbb{P}\left(\bigcup_{i=l}^\infty\left\{\omega\in\Omega\relmiddle|d\left(\xi_i(\omega),\xi(\omega)\right)>\epsilon\right\}\right)=0
\end{equation*}
for any $\epsilon>0$.
If $\{\xi_i\}$ converges completely to $\xi$, then $\{\xi_i\}$ converges almost surely to $\xi$.
\end{remark}


\begin{theorem}\label{v_conv}
As $N$ tends to $\infty$, $\wh{\bs{v}}_N$ converges completely to $\bs{c}$.
\end{theorem}

\begin{theorem}\label{rho_conv}
As $N$ tends to $\infty$, $\wh{\rho}_{\wh{\bs{v}}_N,N}$ converges completely to $\rho(\pi_1f_1, \pi_2f_2)$.
\end{theorem}

The proofs of \Cref{v_conv,rho_conv} are given in \Cref{AppA}.
While $\wh{\bs{v}}_N$ and $\wh{\rho}_{\wh{\bs{v}}_N,N}$ are treated as random variables, their measurability is in fact nontrivial and will be discussed in \Cref{measurability}.

\begin{remark}\label{ind_N}
For each $N=1,2,\ldots$, let $(X_1^{(N)},Y_1^{(N)}),\ldots,(X_N^{(N)},Y_N^{(N)})$ be i.i.d.\ random variables with the distribution $P$ on $\R\times\{1,2\}$ to calculate $\wh{\bs{v}}_N^{(N)}\in\wh{\R}_N^n$ and $\wh{\rho}_{\wh{\bs{v}}_N^{(N)},N}^{(N)}$ in the same way as $\wh{\bs{v}}_N$ and $\wh{\rho}_{\wh{\bs{v}}_N,N}$ in \Cref{hR}, respectively.
By \Cref{v_conv,rho_conv}, we have
\begin{equation*}
\sum_{N=1}^\infty\mathbb{P}\left(\left\{\omega\in\Omega\relmiddle| \left\|\wh{\bs{v}}_N^{(N)}-\bs{c}\right\|>\epsilon\right\}\right)
=\sum_{N=1}^\infty\mathbb{P}\left(\left\{\omega\in\Omega\relmiddle| \left\|\wh{\bs{v}}_N-\bs{c}\right\|>\epsilon\right\}\right)<\infty
\end{equation*}
and
\begin{align*}
&\sum_{N=1}^\infty\mathbb{P}\left(\left\{\omega\in\Omega\relmiddle| \left|\wh{\rho}_{\wh{\bs{v}}_N^{(N)},N}^{(N)}-\rho(\pi_1f_1,\pi_2f_2)\right|>\epsilon\right\}\right)\\
&=\sum_{N=1}^\infty\mathbb{P}\left(\left\{\omega\in\Omega\relmiddle| \left|\wh{\rho}_{\wh{\bs{v}}_N,N}-\rho(\pi_1f_1,\pi_2f_2)\right|>\epsilon\right\}\right)<\infty
\end{align*}
for any $\epsilon>0$, where $\|\cdot\|$ denotes the Euclidean norm.
Hence $\wh{\bs{v}}_N^{(N)}$ and $\wh{\rho}_{\wh{\bs{v}}_N^{(N)},N}^{(N)}$, as well as $\wh{\bs{v}}_N$ and $\wh{\rho}_{\wh{\bs{v}}_N,N}$, converge completely to $\bs{c}$ and $\rho(\pi_1f_1,\pi_2f_2)$, respectively.
\end{remark}

\section{Numerical experiments}
\label{sec:exper}
Here we perform numerical simulations to illustrate the results in \Cref{sec:numer}.
A set of random samples $\{(X_i, Y_i)\mid 1\le i\le N\}$ was simulated under the following two conditions: first, 
\begin{align*}
\pi_1 = 2/3,\qquad \pi_2 = 1/3,\qquad f_1 = \nu_{-1,1},\qquad f_2 = \nu_{1,1},
\end{align*}
and second, 
\begin{align*}
\pi_1 = \pi_2 = 0.5,\qquad f_1 = 0.5\nu_{-1,1}+0.5\nu_{1,1},\qquad f_2 = 0.8\nu_{0,1}+0.2\tau_{0, 0.5},
\end{align*}
where $\nu_{\mu,\sigma}$ represents the Gaussian PDF defined as
\begin{equation}\label{gausspdf}
\nu_{\mu,\sigma}(x) = \frac{1}{\sqrt{2\pi}\sigma}\exp\left(-\frac{(x-\mu)^2}{2\sigma^2} \right)\qquad (x\in\R)
\end{equation}
and $\tau_{a,b}$ is the triangular PDF defined as 
\begin{equation}\label{trianglepdf}
\tau_{a,b}(x) = \begin{cases}
\ 4(x-a)/(b-a)^2 & \mbox{if}\quad a\le x\le (a+b)/2,\\
\ 4(b-x)/(b-a)^2 & \mbox{if}\quad (a+b)/2< x\le b, \\
\ 0 & \mbox{otherwise}.
\end{cases}\qquad (x\in\R;\ a<b).
\end{equation}
Then, we can analytically calculate
\begin{align}
&C(\pi_1f_1, \pi_2f_2) = \{c_1\} = \left\{(\log 2)/2\right\}\simeq \{0.347\}\label{first_C},\\
&\rho(\pi_1f_1, \pi_2f_2) = \left[2-2\Phi\left(c_1+1\right) + \Phi\left(c_1-1\right)\right]/3 \simeq 0.145\label{first_rho}
\end{align}
for the first case, and 
\begin{align}
&C(\pi_1f_1, \pi_2f_2) = \{c_1, c_2\} = \cosh^{-1}\left(0.8\sqrt{\mathrm{e}}\right) \simeq \{-0.779, 0.779\}\label{second_C}, \\
&\rho(\pi_1f_1, \pi_2f_2) = 0.8 - 0.5\Phi(c_1+1) + 0.5\Phi(c_2+1) - 0.8\Phi(c_2) \simeq 0.362\label{second_rho}
\end{align}
for the second case, where $\Phi$ denotes the cumulative distribution function of the standard normal distribution given by
\begin{equation}\label{Phidef}
\Phi(x) = \frac{1}{\sqrt{2\pi}}\int_{-\infty}^x \exp\left(-\frac{t^2}{2} \right)\: dt \qquad (x\in\R).
\end{equation}
See \Cref{AppB} for the proof of \Cref{first_C,first_rho,second_C,second_rho}.
With the knowledge that $n=1$ and $n=2$ for the first and second cases, respectively, we numerically calculated $\wh{\bs{v}}_N$ and $\wh{\rho}_{\wh{\bs{v}}_N,N}$ for each case with $N=10000$.
The subsets $\{(X_i, Y_i)\mid 1\le i\le 100\}$ and $\{(X_i, Y_i)\mid 1\le i\le 1000\}$ were also applied to calculate $\wh{\bs{v}}_N$ and $\wh{\rho}_{\wh{\bs{v}}_N,N}$.
This trial (from the generation of 10000 random samples) was repeated independently for 30 times, and the convergence of $\wh{\bs{v}}_N$ and $\wh{\rho}_{\wh{\bs{v}}_N,N}$ was visually assessed.

\begin{figure}[htpb]
\centering
\includegraphics[width=0.84\textwidth]{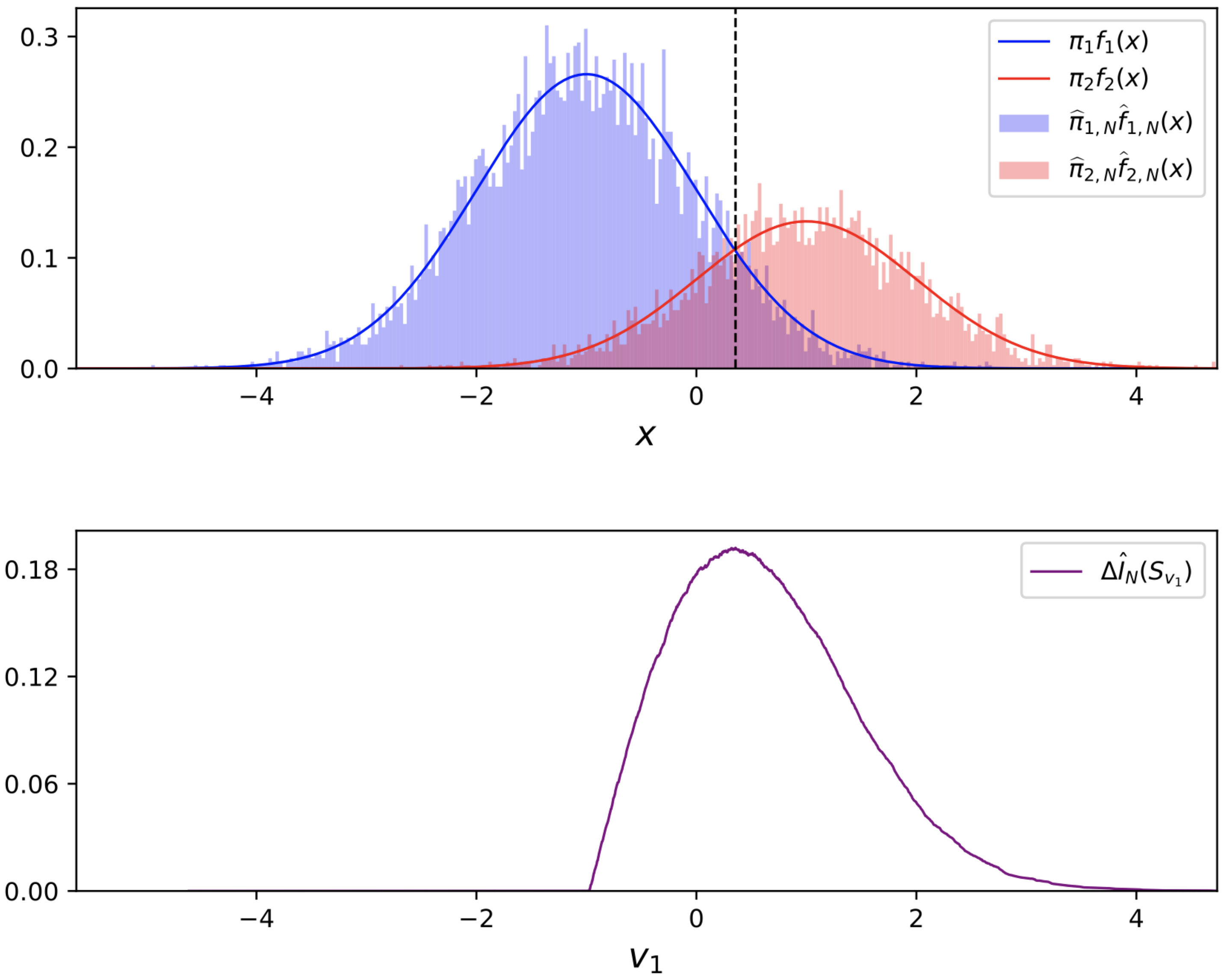}
\caption{In the upper row, $\pi_1f_1$ and $\pi_2f_2$ for the first case are plotted. Normalized histograms corresponding to $\pi_1f_1$ and $\pi_2f_2$ (denoted by $\wh{\pi}_{1,N}\wh{f}_{1,N}$ and $\wh{\pi}_{2,N}\wh{f}_{2,N}$, respectively) were generated using a representative set of $N=10000$ random samples, $\{(X_i, Y_i)\mid 1\le i\le 10000\}$. The vertical dotted line indicates the estimated crossover point $\wh{v}_{1,N}\approx 0.355$, where its theoretical counterpart is $c_1\simeq 0.347$. In the lower row, $\varDelta \wh{I}_N(S_{v_1})$ for all $v_1\in \wh{\R}_N^1$ are plotted. The overlap $\rho(\pi_1f_1, \pi_2f_2)\simeq 0.145$ was estimated as $\wh{\rho}_{\wh{\bs{v}}_N,N}\approx 0.140$.}
\label{fig:single_bin}
\end{figure}

\begin{figure}[htpb]
\centering
\includegraphics[width=0.84\textwidth]{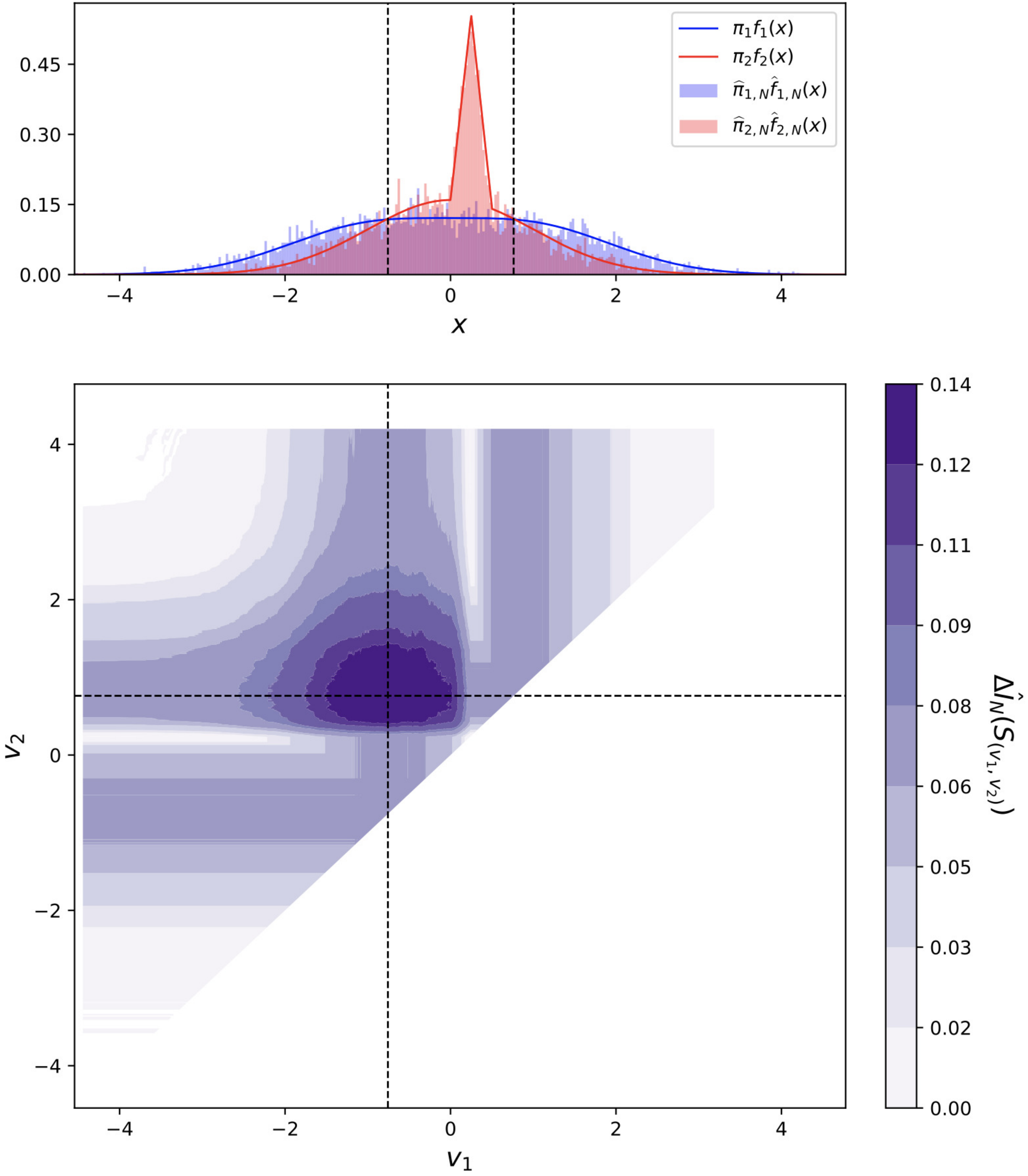}
\caption{In the upper row, $\pi_1f_1$ and $\pi_2f_2$ for the second case are plotted. Normalized histograms corresponding to $\pi_1f_1$ and $\pi_2f_2$ (denoted by $\wh{\pi}_{1,N}\wh{f}_{1,N}$ and $\wh{\pi}_{2,N}\wh{f}_{2,N}$, respectively) were generated using a representative set of $N=10000$ random samples, $\{(X_i, Y_i)\mid 1\le i\le 10000\}$. The dotted lines indicate the estimated crossover points $\wh{v}_{1,N}\approx -0.757$ and $\wh{v}_{2,N}\approx 0.763$, where their theoretical counterparts are $c_1\simeq -0.779$ and $c_2\simeq 0.779$, respectively. In the lower row, $\varDelta \wh{I}_N(S_{(v_1, v_2)})$ for all $(v_1, v_2)\in \wh{\R}_N^2$ are visualized in a heatmap. The overlap $\rho(\pi_1f_1, \pi_2f_2)\simeq 0.362$ was estimated as $\wh{\rho}_{\wh{\bs{v}}_N,N}\approx 0.361$.}
\label{fig:single_ter}
\end{figure}

To begin with, we exhibit a representative sample distribution ($N=10000$) for each case with the calculated values of $\wh{\bs{v}}_N$ and $\wh{\rho}_{\wh{\bs{v}}_N,N}$ (\Cref{fig:single_bin,fig:single_ter}).
As a result of the 30 trials for each case, $\wh{\bs{v}}_N$ and $\wh{\rho}_{\wh{\bs{v}}_N,N}$ appear to converge to $\bs{c}$ and $\rho(\pi_1f_1,\pi_2f_2)$, respectively, as $N$ increases (\Cref{fig:multseq_bin,fig:multseq_ter}).

Similarly, we next performed 30 independent trials for each case to simulate three independent sets of random samples, of the forms $\{(X_i, Y_i)\mid 1\le i\le 100\}$, $\{(X_i', Y_i')\mid 1\le i\le 1000\}$, and $\{(X_i'', Y_i'')\mid 1\le i\le 10000\}$. 
Each set was used to calculate $\wh{\bs{v}}_N^{(N)}$ and $\wh{\rho}_{\wh{\bs{v}}_N^{(N)},N}^{(N)}$ (see \Cref{ind_N}).
Then, in both the cases, $\wh{\bs{v}}_N^{(N)}$ and $\wh{\rho}_{\wh{\bs{v}}_N^{(N)},N}^{(N)}$ appear to converge to $\bs{c}$ and $\rho(\pi_1f_1,\pi_2f_2)$, respectively, as $N$ increases (\Cref{fig:multind_bin,fig:multind_ter}).

\begin{figure}[htpb]
\centering
\includegraphics[width=0.95\textwidth]{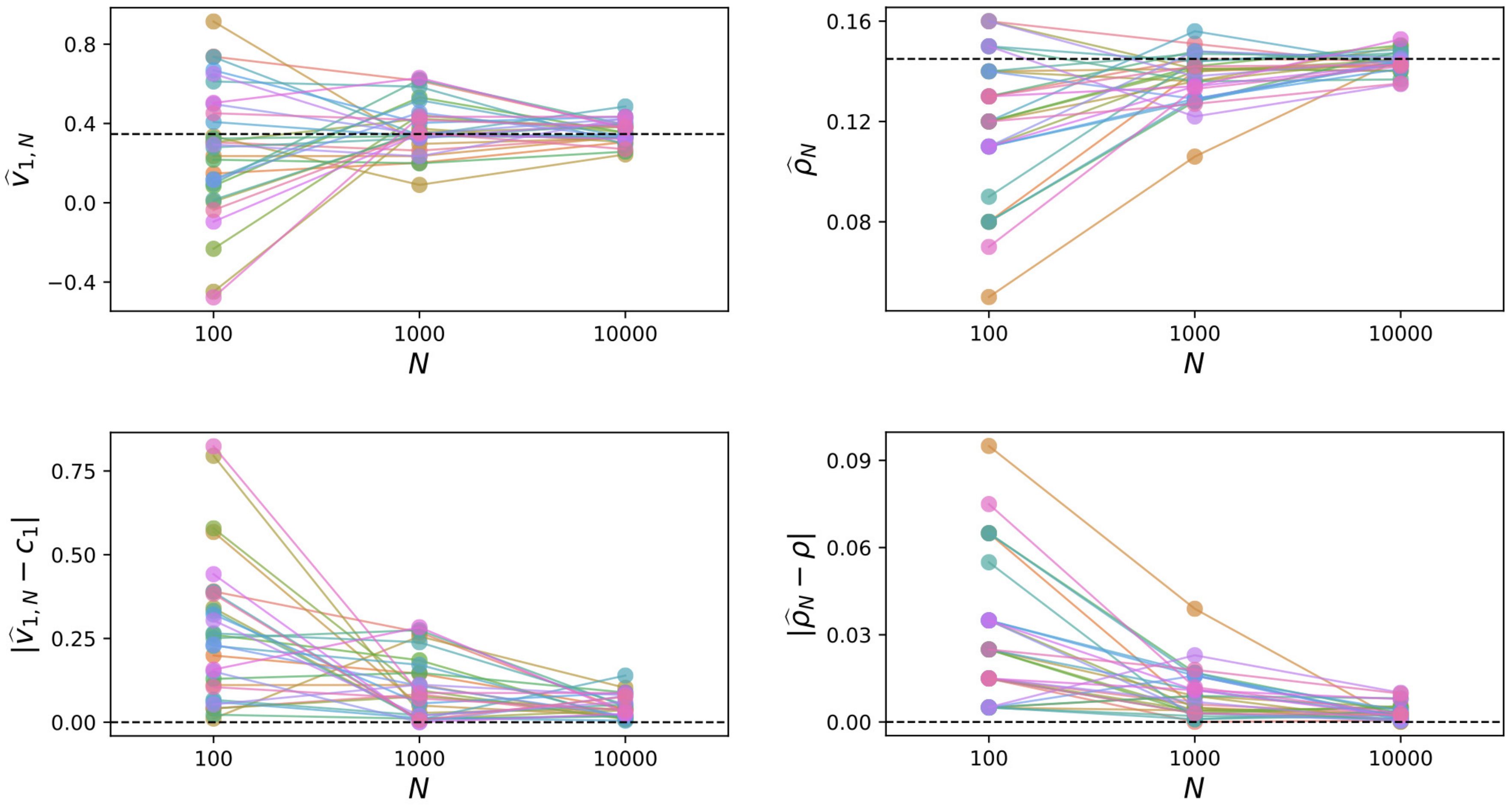}
\caption{In the first case, $30$ independent trials were performed to simulate $10000$ random samples: $(X_1, Y_1),\ldots,(X_{10000}, Y_{10000})$. For each trial, $\{(X_i, Y_i)\mid 1\le i\le 100\}$, $\{(X_i, Y_i)\mid 1\le i\le 1000\}$, and $\{(X_i, Y_i)\mid 1\le i\le 10000\}$ were used to calculate $\wh{v}_{1,N}$, $|\wh{v}_{1,N}-c_1|$, $\wh{\rho}_{\wh{\bs{v}}_N,N}$, and $|\wh{\rho}_{\wh{\bs{v}}_N,N}-\rho(\pi_1f_1,\pi_2f_2)|$. Each dotted line indicates the expected value: $c_1\simeq 0.347$ for $\wh{v}_{1,N}$, $0$ for $|\wh{v}_{1,N}-c_1|$, $\rho(\pi_1f_1, \pi_2f_2)\simeq 0.145$ for $\wh{\rho}_{\wh{\bs{v}}_N,N}$, and $0$ for $|\wh{\rho}_{\wh{\bs{v}}_N,N}-\rho(\pi_1f_1,\pi_2f_2)|$. In this figure, $\wh{\rho}_{\wh{\bs{v}}_N,N}$ and $\rho(\pi_1f_1,\pi_2f_2)$ are abbreviated as $\wh{\rho}_N$ and $\rho$, respectively.}
\label{fig:multseq_bin}
\end{figure}

\begin{figure}[htpb]
\centering
\includegraphics[width=0.95\textwidth]{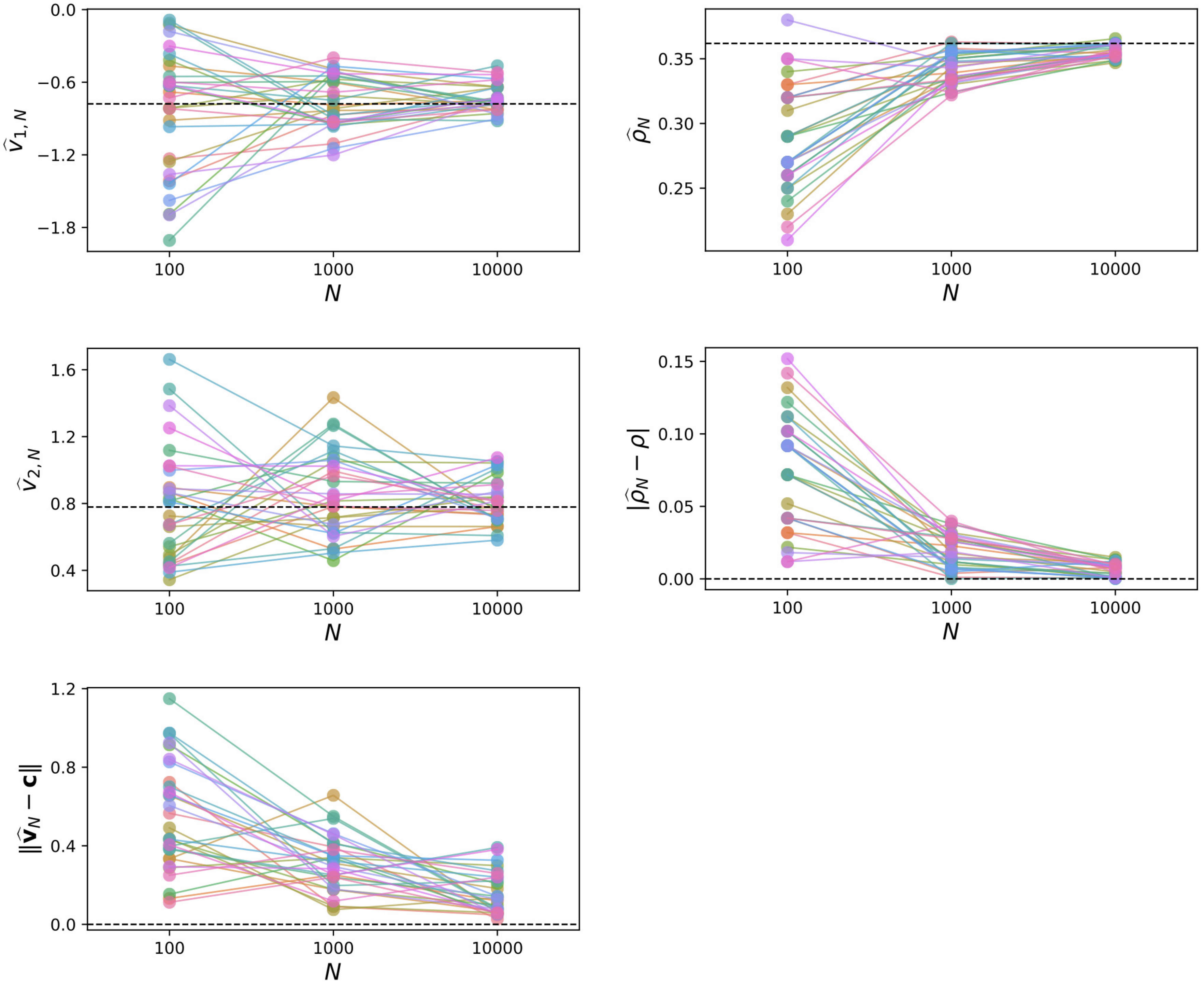}
\caption{In the second case, $30$ independent trials were performed to simulate $10000$ random samples: $(X_1, Y_1),\ldots,(X_{10000}, Y_{10000})$. For each trial, $\{(X_i, Y_i)\mid 1\le i\le 100\}$, $\{(X_i, Y_i)\mid 1\le i\le 1000\}$, and $\{(X_i, Y_i)\mid 1\le i\le 10000\}$ were used to calculate $\wh{v}_{1,N}$, $\wh{v}_{2,N}$, $\|\wh{\bs{v}}_N-\bs{c}\|$, $\wh{\rho}_{\wh{\bs{v}}_N,N}$, and $|\wh{\rho}_{\wh{\bs{v}}_N,N}-\rho(\pi_1f_1,\pi_2f_2)|$. Each dotted line indicates the expected value: $c_1\simeq -0.779$ for $\wh{v}_{1,N}$, $c_2\simeq 0.779$ for $\wh{v}_{2,N}$, $0$ for $\|\wh{\bs{v}}_N-\bs{c}\|$, $\rho(\pi_1f_1, \pi_2f_2)\simeq 0.362$ for $\wh{\rho}_{\wh{\bs{v}}_N,N}$, and $0$ for $|\wh{\rho}_{\wh{\bs{v}}_N,N}-\rho(\pi_1f_1,\pi_2f_2)|$. In this figure, $\wh{\rho}_{\wh{\bs{v}}_N,N}$ and $\rho(\pi_1f_1,\pi_2f_2)$ are abbreviated as $\wh{\rho}_N$ and $\rho$, respectively.}
\label{fig:multseq_ter}
\end{figure}

\begin{figure}[htpb]
\centering
\includegraphics[width=0.95\textwidth]{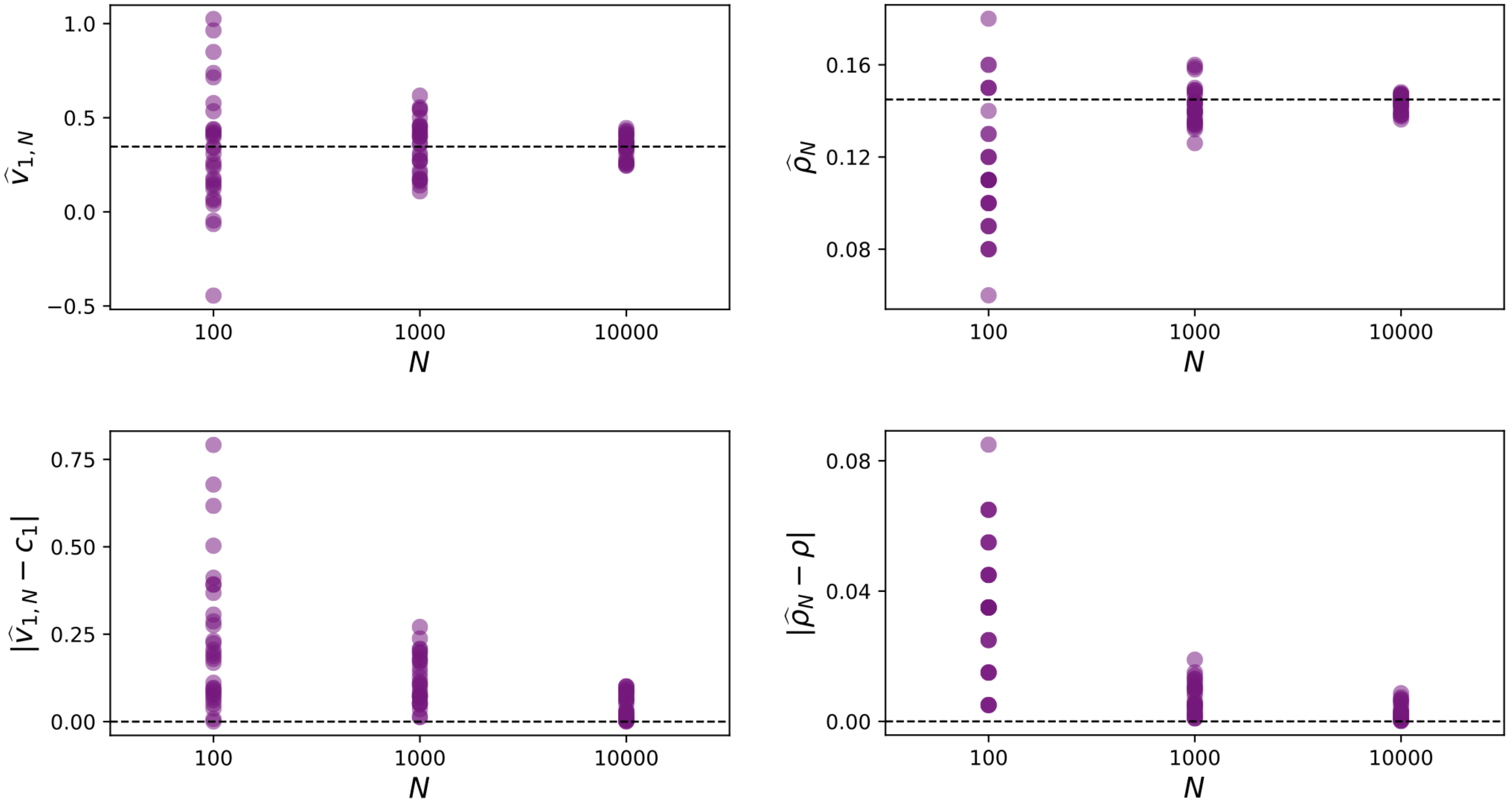}
\caption{In the first case, $30$ independent trials were performed to simulate three independent sets of random samples, of the forms $\{(X_i, Y_i)\mid 1\le i\le 100\}$, $\{(X_i', Y_i')\mid 1\le i\le 1000\}$, and $\{(X_i'', Y_i'')\mid 1\le i\le 10000\}$. Each set was used to calculate $\wh{v}_{1,N}$, $|\wh{v}_{1,N}-c_1|$, $\wh{\rho}_{\wh{\bs{v}}_N,N}$, and $|\wh{\rho}_{\wh{\bs{v}}_N,N}-\rho(\pi_1f_1,\pi_2f_2)|$. Note that the superscript $(N)$ in \Cref{ind_N} is omitted here. The dotted lines indicate the expected values: $c_1\simeq 0.347$ for $\wh{v}_{1,N}$, $0$ for $|\wh{v}_{1,N}-c_1|$, $\rho(\pi_1f_1, \pi_2f_2)\simeq 0.145$ for $\wh{\rho}_{\wh{\bs{v}}_N,N}$, and $0$ for $|\wh{\rho}_{\wh{\bs{v}}_N,N}-\rho(\pi_1f_1,\pi_2f_2)|$. In this figure, $\wh{\rho}_{\wh{\bs{v}}_N,N}$ and $\rho(\pi_1f_1,\pi_2f_2)$ are abbreviated as $\wh{\rho}_N$ and $\rho$, respectively.}
\label{fig:multind_bin}
\end{figure}

\begin{figure}[htpb]
\centering
\includegraphics[width=0.95\textwidth]{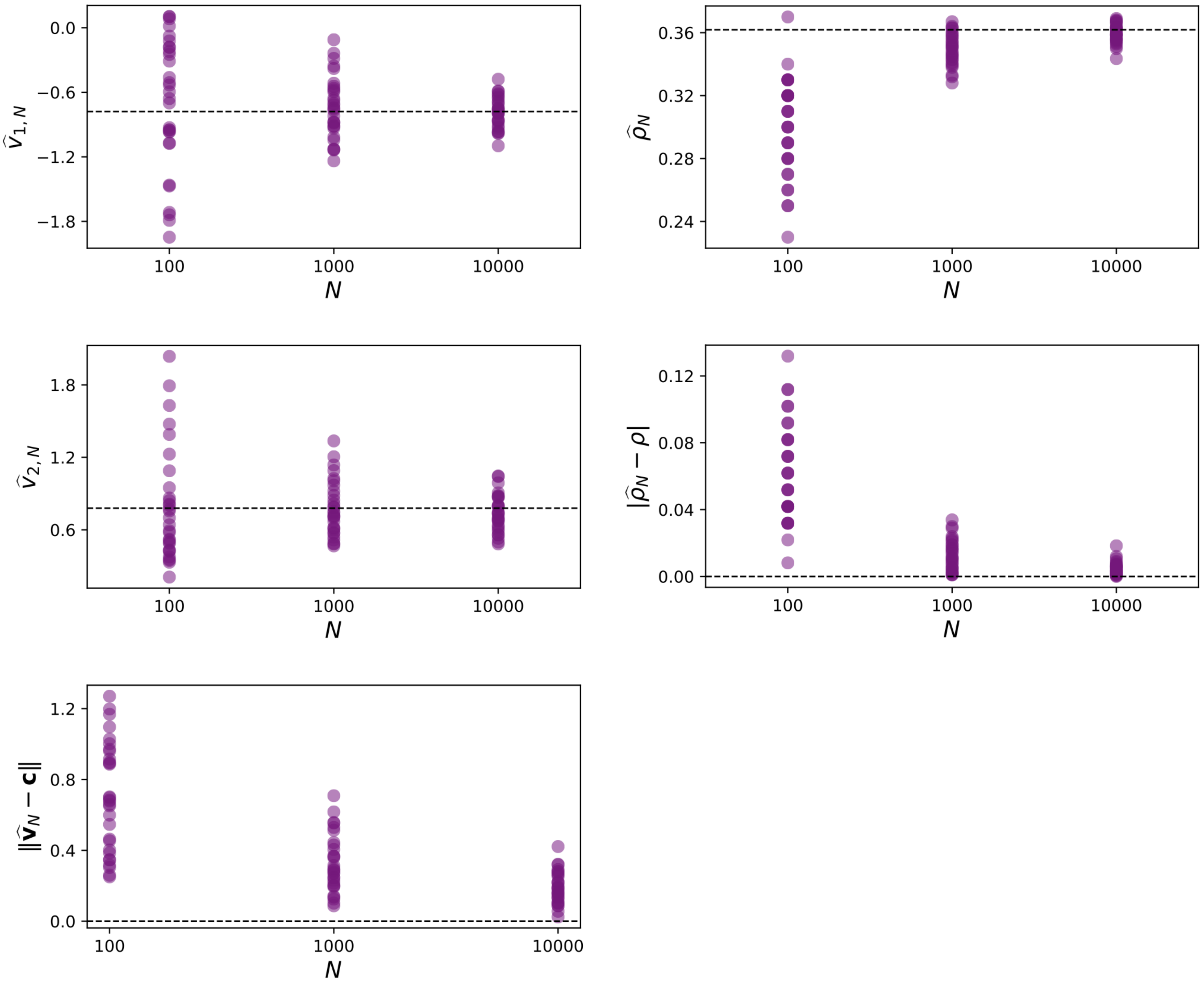}
\caption{In the second case, $30$ independent trials were performed to simulate three independent sets of random samples, of the forms $\{(X_i, Y_i)\mid 1\le i\le 100\}$, $\{(X_i', Y_i')\mid 1\le i\le 1000\}$, and $\{(X_i'', Y_i'')\mid 1\le i\le 10000\}$. Each set was used to calculate $\wh{v}_{1,N}$, $\wh{v}_{2,N}$, $\|\wh{\bs{v}}_N-\bs{c}\|$, $\wh{\rho}_{\wh{\bs{v}}_N,N}$, and $|\wh{\rho}_{\wh{\bs{v}}_N,N}-\rho(\pi_1f_1,\pi_2f_2)|$. Note that the superscript $(N)$ in \Cref{ind_N} is omitted here. The dotted lines indicate the expected values: $c_1\simeq -0.779$ for $\wh{v}_{1,N}$, $c_2\simeq 0.779$ for $\wh{v}_{2,N}$, $0$ for $\|\wh{\bs{v}}_N-\bs{c}\|$, $\rho(\pi_1f_1, \pi_2f_2)\simeq 0.362$ for $\wh{\rho}_{\wh{\bs{v}}_N,N}$, and $0$ for $|\wh{\rho}_{\wh{\bs{v}}_N,N}-\rho(\pi_1f_1,\pi_2f_2)|$. In this figure, $\wh{\rho}_{\wh{\bs{v}}_N,N}$ and $\rho(\pi_1f_1,\pi_2f_2)$ are abbreviated as $\wh{\rho}_N$ and $\rho$, respectively.}
\label{fig:multind_ter}
\end{figure}

\section{Conclusion}
\label{sec:conclusion}
In this paper, we propose a new nonparametric framework to calculate OVL based on a decision tree algorithm.
The estimators of crossover points and overlaps for continuous PDFs were shown to converge to the expected values (both analytically and numerically).
However, there remain several issues to be addressed:
\begin{enumerate}
\item We have not established a general way to know the number $n$ of crossover points (which is required to be known in advance), though we may estimate it beforehand by obtaining partial information about the distributions (e.g., there exist precisely two crossover points between any two normal distributions with different variances) or by using some numerical tools like histograms.
\item Our method has not been applied to real data or compared numerically with other nonparametric methods, though the following arguments seem to exemplify the theoretical advantages of ours over the previous ones (described in detail in \cite{schmid06}): (i) our OVL estimator depends only on the rank statistics of $X_1,...,X_N$ (labeled by $Y_1,\ldots,Y_N$, respectively), as is consistent with the nature of OVL, while the OVL estimators in \cite{schmid06} depend not only on the rank statistics (\cite[pp. 1588--1589]{schmid06}); (ii) our OVL estimator converges completely to the true value (\cref{rho_conv}). 
\end{enumerate}
Further studies on these problems are needed for the practical use of our method.

\begin{appendix}
\section{Additional proofs}\label{AppA}
\Cref{v_conv,rho_conv} will be proved in this section.
We shall take over the notations in \Cref{sec:numer} and, in addition, write $h(\bs{v})$ and $\wh{h}_N(\bs{v})$ in place of $\varDelta I(S_{\bs{v}})$ and $\varDelta \wh{I}_N(S_{\bs{v}})$, respectively.

\begin{definition}\label{def:A1}
\upshape
For $j\in\{1,2\}$ and $x\in \R$, define
\begin{equation*}
\wh{F}_{j,N}(x) = \begin{cases}
\ N_{XY}((-\infty, x], j)/N_Y(j) & \mbox{if}\quad N_Y(j)>0,\\
\ 0 & \mbox{if}\quad N_Y(j)=0.
\end{cases}
\end{equation*}
We also define $\wh{F}_{j,N}(-\infty)=0$ and $\wh{F}_{j,N}(\infty)=1$.
\end{definition}

\begin{proposition}\label{hhat}
For $\bs{v}=(v_1,\ldots,v_m)\in\R_\le^m$ with $m$ a positive integer,
\begin{equation*}
\wh{h}_N(\bs{v})= \sum_{k=1}^{m+1}\max_j\left\{\wh{\pi}_{j,N}\left[\wh{F}_{j,N}(v_k)-\wh{F}_{j,N}(v_{k-1})\right] \right\} - \max_j \left\{\wh{\pi}_{j,N}\right\},
\end{equation*}
where $v_0=-\infty$, $v_{m+1}=\infty$, $\wh{F}_{j,N}(v_0)=0$, and $\wh{F}_{j,N}(v_{m+1})=1$.
\end{proposition}
\begin{proof}
From \Cref{delIh}, we have
\begin{equation}\label{hhat2}
\begin{aligned}
\wh{h}_N(\bs{v}) &= \sum_k \wh{P}_N(X\in S_{\bs{v},k})\max_j\left\{\wh{P}_N(Y=j\mid X\in S_{\bs{v},k})\right\}\\ 
&\qquad - \max_j\left\{\wh{P}_N(Y=j\mid X\in\R)\right\},
\end{aligned}
\end{equation}
where the sum is over all $k$ with $N_X(S_{\bs{v},k})>0$.
Since $\wh{P}_N(X\in S_{\bs{v},k})=N_X(S_{\bs{v},k})/N$, $\wh{P}_N(Y=j\mid X\in S_{\bs{v},k}) = N_{XY}(S_{\bs{v},k}, j)/N_X(S_{\bs{v},k})$, and $N_{XY}(S_{\bs{v},k}, j) = N\wh{\pi}_{j,N}[\wh{F}_{j,N}(v_k)-\wh{F}_{j,N}(v_{k-1})]$,
we obtain
\begin{equation*}
\wh{P}_N(X\in S_{\bs{v},k})\max_j\left\{\wh{P}_N(Y=j\mid X\in S_{\bs{v},k})\right\} = \max_j\left\{\wh{\pi}_{j,N}\left[\wh{F}_{j,N}(v_k)-\wh{F}_{j,N}(v_{k-1})\right] \right\}.
\end{equation*}
As for the last term of \Cref{hhat2}, we have $\wh{P}_N(Y=j\mid X\in\R)=\wh{\pi}_{j,N}$ by definition.
\end{proof}

\begin{corollary}
For $\bs{v}\in\R_\le^m$ with $m$ a positive integer, $\wh{h}_N(\bs{v})\ge 0$ and $h(\bs{v})\ge 0$.
\end{corollary}

\begin{proof}
Let $\wh{\pi}_{p,N} = \max\,\{\wh{\pi}_{1,N}, \wh{\pi}_{2,N}\}$.
By \Cref{hhat}, we have
\begin{align*}
\wh{h}_N(\bs{v}) &= \sum_{k=1}^{m+1}\max_j\left\{\wh{\pi}_{j,N}\left[\wh{F}_{j,N}(v_k)-\wh{F}_{j,N}(v_{k-1})\right] \right\} - \max_j \left\{\wh{\pi}_{j,N}\right\}\\
&\ge \sum_{k=1}^{m+1}\wh{\pi}_{p,N} \left[\wh{F}_{p,N}(v_k)-\wh{F}_{p,N}(v_{k-1}) \right] - \wh{\pi}_{p,N} = 0.
\end{align*}
We can similarly prove that $h(\bs{v})\ge 0$ from \Cref{delI_F}.
\end{proof}

For simplicity, we may write $\varphi_j(v,v')$ and $\wh{\varphi}_{j,N}(v,v')$ in place of $\pi_j[F_j(v)-F_j(v')]$ and $\wh{\pi}_{j,N}[\wh{F}_{j,N}(v)-\wh{F}_{j,N}(v')]$, respectively, so that
\begin{align}
&h(\bs{v}) = \sum_{k=1}^{m+1}\max_j\left\{\varphi_j(v_k, v_{k-1})\right\} - \max_j \left\{\pi_j\right\},\label{hveq}\\
&\wh{h}_N(\bs{v}) = \sum_{k=1}^{m+1}\max_j\left\{\wh{\varphi}_{j,N}(v_k, v_{k-1}) \right\} - \max_j \left\{\wh{\pi}_{j,N}\right\}\label{hhveq}
\end{align}
by \Cref{delI_F,hhat}.

\begin{definition}\label{Dm}
\upshape
For $m=1,\ldots,n$, define
\begin{align*}
&\V_m=\argmax_{\bs{v}\in\R_\le^m}\left\{h({\bs{v}})\right\},\\
&\wh{\V}_{m,N}=\argmax_{\bs{v}\in\wh{\R}_N^m}\left\{\wh{h}_N({\bs{v}})\right\},\\
&\C_m=\left\{(c_{i_1},\ldots,c_{i_m})\mid 1\le i_1< \cdots< i_m\le n\right\}.
\end{align*}
\end{definition}

\begin{remark}\label{VmCm}
We will see that $\V_m\ne\emptyset$ ($m\le n$) by \Cref{hwhv2,delISc}.
Since $\wh{\R}_N^m$ is a nonempty finite set (see \Cref{hR}), $\wh{\V}_{m,N}\ne\emptyset$.
\end{remark}

\begin{proposition}\label{hwhv}
Let $m$ be a positive integer with $m<n$.
Then for any $\bs{v}=(v_1,\ldots,v_m)\in\R_\le^m$, there exists $\bs{w}=(c_{i_1},\ldots,c_{i_m})$ with $1\le i_1\le\cdots\le i_m\le n$ such that $h(\bs{w})\ge h(\bs{v})$.
\end{proposition}

\begin{proof}
Let $\bs{v}=(v_1,\ldots,v_m)\in\R_\le^m$ be given.
Set $v_0=-\infty$, $v_{m+1}=\infty$, and $r(\bs{v})=\#\{k\in\{1,\ldots,m\}\mid v_k\notin C(\pi_1f_1, \pi_2f_2)\}$.
The statement obviously holds when $r(\bs{v})=0$.

Let $r(\bs{v})>0$.
Then we can choose $v_p\notin C(\pi_1f_1, \pi_2f_2)$ ($1\le p\le m$) and $c_q\in C(\pi_1f_1, \pi_2f_2)$ ($1\le q\le n$) satisfying $c_{q-1}<v_p<c_q\le v_{p+1}$ or $v_{p-1}\le c_q<v_p<c_{q+1}$.
We will only show the case $c_{q-1}<v_p<c_q\le v_{p+1}$, as the other is similar.
Without loss of generality, we may assume that $\pi_1f_1\ge \pi_2f_2$ on $(v_p, c_q)$, so that $\varphi_1(c_q, v_p)> \varphi_2(c_q, v_p)$ and $\varphi_1(v_p, c_{q-1})> \varphi_2(v_p, c_{q-1})$, since $C'(\pi_1f_1, \pi_2f_2)$ is finite.
In the following, we consider the cases (I) $\varphi_1(v_p, v_{p-1})\ge \varphi_2(v_p, v_{p-1})$ and (II) $\varphi_1(v_p, v_{p-1})< \varphi_2(v_p, v_{p-1})$.

(I) Suppose $\varphi_1(v_p, v_{p-1})\ge \varphi_2(v_p, v_{p-1})$. Then 
\begin{align*}
&\varphi_1(c_q, v_{p-1}) > \varphi_2(c_q, v_{p-1}), \\
&\varphi_j(v_{p+1}, c_q) = \varphi_j(v_{p+1}, v_p) - \varphi_j(c_q, v_p) \qquad (j=1,2), \\
&\varphi_j(c_q, v_{p-1}) = \varphi_j(v_p, v_{p-1}) + \varphi_j(c_q, v_p) \qquad (j=1,2),
\end{align*}
hence
\begin{align*}
&\max_j\left\{\varphi_j(c_q, v_{p-1})\right\} + \max_j\left\{\varphi_j(v_{p+1}, c_q)\right\} \\
&= \varphi_1(c_q, v_{p-1}) + \max_j \left\{\varphi_j(v_{p+1}, v_p) - \varphi_j(c_q, v_p) \right\} \\
&= \varphi_1(v_p, v_{p-1}) + \varphi_1(c_q, v_p) + \max_j \left\{\varphi_j(v_{p+1}, v_p) - \varphi_j(c_q, v_p) \right\} \\
&\ge \varphi_1(v_p, v_{p-1}) + \varphi_1(c_q, v_p) + \max_j \left\{\varphi_j(v_{p+1}, v_p)\right\} - \varphi_1(c_q, v_p) \\
&= \varphi_1(v_p, v_{p-1}) + \max_j \left\{\varphi_j(v_{p+1}, v_p)\right\} \\
&= \max_j \left\{\varphi_j(v_p, v_{p-1})\right\} + \max_j \left\{\varphi_j(v_{p+1}, v_p)\right\},
\end{align*}
and setting $\bs{v}'=(v_1,\ldots,v_{p-1},c_q,v_{p+1},\ldots,v_m)\in\R_\le^m$ gives $r(\bs{v}')<r(\bs{v})$ and $h(\bs{v}')\ge h(\bs{v})$.

(II) Suppose $\varphi_1(v_p, v_{p-1})< \varphi_2(v_p, v_{p-1})$.
Since $\pi_1f_1\ge \pi_2f_2$ on $(v_p, c_q)$, we can see that $v_{p-1}<c_{q-1}<v_p$ and $\varphi_1(c_{q-1}, v_{p-1}) < \varphi_2(c_{q-1}, v_{p-1})$.
First consider the case (II-a) $\varphi_1(v_{p+1}, v_p)\ge \varphi_2(v_{p+1}, v_p)$. 
Then $\varphi_1(v_{p+1}, c_{q-1}) > \varphi_2(v_{p+1}, c_{q-1})$, hence
\begin{align*}
&\max_j\left\{\varphi_j(c_{q-1}, v_{p-1})\right\} + \max_j\left\{\varphi_j(v_{p+1}, c_{q-1})\right\} \\
&= \varphi_2(c_{q-1}, v_{p-1}) + \varphi_1(v_{p+1}, c_{q-1})\\
&= \varphi_2(c_{q-1}, v_{p-1}) + \varphi_1(v_{p+1}, v_p) + \varphi_1(v_p, c_{q-1})\\
&> \varphi_2(c_{q-1}, v_{p-1}) + \varphi_1(v_{p+1}, v_p) + \varphi_2(v_p, c_{q-1})\\
&= \varphi_2(v_p, v_{p-1}) + \varphi_1(v_{p+1}, v_p)\\
&= \max_j \left\{\varphi_j(v_p, v_{p-1})\right\} + \max_j \left\{\varphi_j(v_{p+1}, v_p)\right\},
\end{align*}
and setting $\bs{v}'=(v_1,\ldots,v_{p-1},c_{q-1},v_{p+1},\ldots,v_m)\in\R_\le^m$ gives $r(\bs{v}')<r(\bs{v})$ and $h(\bs{v}')> h(\bs{v})$.
Next consider the case (II-b) $\varphi_1(v_{p+1}, v_p)< \varphi_2(v_{p+1}, v_p)$.
If there exists $x\in (c_{q-1}, v_p)$ such that $\varphi_1(v_{p+1}, x) \ge \varphi_2(v_{p+1}, x)$, then $\varphi_1(x, v_{p-1})<\varphi_2(x, v_{p-1})$, hence the case (II-a) applies to $\bs{v}'' = (v_1,\ldots,v_{p-1},x,v_{p+1},\ldots,v_m)\in\R_\le^m$, where $r(\bs{v}'') = r(\bs{v})$ and
\begin{align*}
&h(\bs{v}'') - h(\bs{v}) \\
&= \max_j\,\{\varphi_j(x, v_{p-1})\} + \max_j\,\{\varphi_j(v_{p+1}, x)\} - \max_j\,\{\varphi_j(v_p, v_{p-1})\} - \max_j\,\{\varphi_j(v_{p+1}, v_p)\} \\
&= \varphi_2(x, v_{p-1}) + \varphi_1(v_{p+1}, x) - \varphi_2(v_p, v_{p-1}) - \varphi_2(v_{p+1}, v_p) \\
&\ge \varphi_2(x, v_{p-1}) + \varphi_2(v_{p+1}, x) - \varphi_2(v_p, v_{p-1}) - \varphi_2(v_{p+1}, v_p) \\
&= \varphi_2(v_{p+1}, v_{p-1}) - \varphi_2(v_{p+1}, v_{p-1}) = 0.
\end{align*}
If $\varphi_1(v_{p+1}, x) < \varphi_2(v_{p+1}, x)$ for any $x\in (c_{q-1}, v_p)$, then $\varphi_1(v_{p+1}, c_{q-1})\le\varphi_2(v_{p+1}, c_{q-1})$, and setting $\bs{v}'=(v_1,\ldots,v_{p-1},c_{q-1},v_{p+1},\ldots,v_m)\in\R_\le^m$ gives $r(\bs{v}')<r(\bs{v})$ and 
\begin{align*}
&h(\bs{v}') - h(\bs{v}) \\
&= \max_j\,\{\varphi_j(c_{q-1}, v_{p-1})\} + \max_j\,\{\varphi_j(v_{p+1}, c_{q-1})\} - \max_j\,\{\varphi_j(v_p, v_{p-1})\} - \max_j\,\{\varphi_j(v_{p+1}, v_p)\} \\
&= \varphi_2(c_{q-1}, v_{p-1}) + \varphi_2(v_{p+1}, c_{q-1}) - \varphi_2(v_p, v_{p-1}) - \varphi_2(v_{p+1}, v_p) \\
&= \varphi_2(v_{p+1}, v_{p-1}) - \varphi_2(v_{p+1}, v_{p-1}) = 0.
\end{align*}

Taken together, for any $\bs{v}\in\R_\le^m$ with $r(\bs{v})>0$, there exists $\bs{v}'\in\R_\le^m$ such that $r(\bs{v}') < r(\bs{v})$ and $h(\bs{v}')\ge h(\bs{v})$.
The statement follows by induction.
\end{proof}

\begin{corollary}\label{hwhv2}
If $m$ is a positive integer with $m<n$, then there exists $\bs{c}'\in\C_m$ such that $h(\bs{c}')=\sup\,\{h(\bs{v})\mid \bs{v}\in\R_\le^m\}$.
Furthermore, $h(\bs{c}')<h(\bs{c})$.
\end{corollary}
\begin{proof}
Since there are only finitely many choices for $\bs{w}\in\R_\le^m$ in \Cref{hwhv}, we can choose $\bs{w}'=(c_{i_1},\ldots,c_{i_m})\in\argmax_{\bs{w}}h(\bs{w})$, where $\bs{w}$ ranges over the choices.
Then $h(\bs{w}')\ge h(\bs{v})$ for all $\bs{v}\in\R_\le^m$.
Let $A=\{c_{i_1},\ldots,c_{i_m}\}$ and assume that $\bs{w}'\notin\C_m$.
Then $\#A<m$, and there exists $A'=\{c_{j_1},\ldots,c_{j_m}\}$ such that $A\subset A'$ and $1\le j_1<\cdots<j_m\le n$.
Put $\bs{c}'=(c_{j_1},\ldots,c_{j_m})$.
Then $\bs{c}'\in\C_m$, and we can see that $h(\bs{c}')\ge h(\bs{w}')$ by definition.
Furthermore, $h(\bs{c}')<h(\bs{c})$ by \Cref{delISc2}.
\end{proof}


%

\begin{remark}
Note that $\bs{v}\in\V_m$ does not necessarily imply $\bs{v}\in\C_m$.
Here we give an example for the case where $(n,m)=(2,1)$ and $\V_1\not\subset\C_1$.
Assume that $\pi_1=0.9$, $\pi_2=0.1$, $f_1=\nu_{0,1}$, and $f_2=\tau_{-0.1,0.1}$ (see \Cref{gausspdf,trianglepdf} for the definitions of $\nu$ and $\tau$).
Then $\pi_1f_1(0)<\pi_2f_2(0)$, $n=2$, and $\C_2 = \{c_1, c_2\}$ where $-0.1<c_1<0<c_2<0.1$.
Since $\varphi_1(\infty, 0.1)=\varphi_1(-0.1, -\infty)=\pi_1\Phi(-0.1)\simeq 0.4142>\pi_2$ (see \Cref{Phidef} for the definition of $\Phi$), $\varphi_1(v, -\infty)>\varphi_2(v, -\infty)$ and $\varphi_1(\infty, v)>\varphi_2(\infty, v)$ hold for all $v\in\R$.
Hence $h(v)=\pi_1$ for all $v\in\R$, and therefore $\V_1=\R\not\subset\{c_1,c_2\}=\C_1$.
\end{remark}

For a real random variable $\xi$ on $(\Omega,\mathcal{F},\mathbb{P})$, we denote its expectation and variance by
\begin{equation*}
\E\,[\xi] = \int_\Omega \xi\:d\mathbb{P},\qquad
\Var\,[\xi] = \int_\Omega\left(\xi-\E\,[\xi]\right)^2\,d\mathbb{P},
\end{equation*}
respectively.
We also denote by $\1_A$ the indicator function of a set $A$, i.e.,
\begin{equation*}
\1_A(t) = \begin{cases}
\ 1 & \mbox{if}\quad t\in A,\\
\ 0 & \mbox{if}\quad t\notin A.
\end{cases}
\end{equation*}

\begin{theorem}\label{slln}
{\normalfont (Kolmogorov's strong law of large numbers. See \cite{hsu47} for the proof.)}
Let $\{\xi_i\}$ be a sequence of i.i.d.\ real random variables on $(\Omega,\mathcal{F},\mathbb{P})$ with $\E\,[|\xi_1|]<\infty$ and $\Var\,[\xi_1]<\infty$.
Let $\mu=\E\,[\xi_1]$ and $s_k=\xi_1+\cdots+\xi_k$ (k=1,2,\ldots).
Then $s_k/k$ converges completely to $\mu$.
\end{theorem}

\begin{theorem}\label{Fhconv}
{\normalfont (The Glivenko-Cantelli theorem. See \cite[Theorem A, Section 2.1.4]{serfling80} for the proof.)}
For each $j\in\{1,2\}$, $\sup_{x\in\R}|\wh{F}_{j,N}(x)-F_j(x)|$
converges completely to $0$ as $N\to\infty$.
\end{theorem}

\begin{proposition}\label{pihconv}
For each $j\in\{1,2\}$, $\wh{\pi}_{j,N}$ converges completely to $\pi_j$ as $N\to\infty$.
\end{proposition}
\begin{proof}
We can see $\1_{\{j\}}(Y_1),\ldots,\1_{\{j\}}(Y_N)$ as i.i.d.\ random variables with $\E\,[\1_{\{j\}}(Y_1)]=\pi_j<\infty$ and $\Var\,[\1_{\{j\}}(Y_1)]=\pi_j(1-\pi_j)<\infty$.
Since $N_Y(j) = \1_{\{j\}}(Y_1)+\cdots+\1_{\{j\}}(Y_N)$, $\wh{\pi}_{j,N}=N_Y(j)/N$ converges completely to $\pi_j$ by \Cref{slln}.
\end{proof}

\begin{lemma}\label{xyzw}
If $x,y,z,w\in\R$, then 
\begin{itemize}
\item[\rm{(a)}]\ $|\max\,\{x,y\}-\max\,\{z,w\}|\le |x-z|+|y-w|$,
\item[\rm{(b)}]\ $|\min\,\{x,y\}-\min\,\{z,w\}|\le |x-z|+|y-w|$.
\end{itemize}
\end{lemma}
\begin{proof}
For (a), suppose $\max\,\{x,y\}\ge\max\,\{z,w\}$ and $x\ge y$ without loss of generality.
If $z\ge w$, then $|\max\,\{x,y\}-\max\,\{z,w\}|=|x-z|\le|x-z|+|y-w|$.
If $z<w$, then $|\max\,\{x,y\}-\max\,\{z,w\}|=|x-w|<|x-z|\le|x-z|+|y-w|$.

For (b), suppose $\min\,\{x,y\}\ge\min\,\{z,w\}$ and $x\ge y$ without loss of generality.
If $z\ge w$, then $|\min\,\{x,y\}-\min\,\{z,w\}|=|y-w|\le|x-z|+|y-w|$.
If $z<w$, then $|\min\,\{x,y\}-\min\,\{z,w\}|=|y-z|\le|x-z|\le|x-z|+|y-w|$.
\end{proof}

\begin{theorem}\label{suphconv}
For any positive integer $m$, $\sup_{\bs{v}\in\R_\le^m}|\wh{h}_N(\bs{v})-h(\bs{v})|$ converges completely to $0$ as $N\to\infty$.
\end{theorem}
\begin{proof}
For all $\bs{v}\in\R_\le^m$, we have
\begin{align*}
&\left|\wh{h}_N(\bs{v}) - h(\bs{v})\right| \\
&\le \sum_{k=1}^{m+1}\left|\max_j\left\{\wh{\varphi}_{j,N}(v_k, v_{k-1}) \right\} - \max_j\left\{\varphi_j(v_k, v_{k-1})\right\}\right| + \left|\max_j \left\{\wh{\pi}_{j,N}\right\} - \max_j \left\{\pi_j\right\}\right| \\
&\le \sum_{k=1}^{m+1}\sum_{j=1}^2 \left|\wh{\varphi}_{j,N}(v_k, v_{k-1}) - \varphi_j(v_k, v_{k-1}) \right| + \sum_{j=1}^2\left|\wh{\pi}_{j,N}-\pi_j \right|
\end{align*}
by \Cref{hveq}, \Cref{hhveq}, and \Cref{xyzw}.
Since
\begin{align*}
&\left|\wh{\varphi}_{j,N}(v_k, v_{k-1}) - \varphi_j(v_k, v_{k-1}) \right| \\
&=\Big| (\wh{\pi}_{j,N}-\pi_j)\left[\wh{F}_{j,N}(v_k)-\wh{F}_{j,N}(v_{k-1})\right]\\
&\qquad + \pi_j\left[\wh{F}_{j,N}(v_k)-F_j(v_k)\right] - \pi_j\left[\wh{F}_{j,N}(v_{k-1})-F_j(v_{k-1})\right]\Big| \\
&\le |\wh{\pi}_{j,N}-\pi_j|\left|\wh{F}_{j,N}(v_k)-\wh{F}_{j,N}(v_{k-1})\right|\\ 
&\qquad + \pi_j\left|\wh{F}_{j,N}(v_k)-F_j(v_k)\right| + \pi_j\left|\wh{F}_{j,N}(v_{k-1})-F_j(v_{k-1})\right| \\
&\le |\wh{\pi}_{j,N}-\pi_j| + 2\pi_j\sup_{x\in\R}\left|\wh{F}_{j,N}(x)-F_j(x)\right|,
\end{align*}
we obtain
\begin{equation*}\label{suphv}
\sup_{\bs{v}\in\R_\le^m}\left|\wh{h}_N(\bs{v}) - h(\bs{v})\right|
\le (m+2)\sum_{j=1}^2\left|\wh{\pi}_{j,N}-\pi_j\right|
+ 2(m+1)\sum_{j=1}^2\pi_j\sup_{x\in\R}\left|\wh{F}_{j,N}(x)-F_j(x)\right|.
\end{equation*}
Hence 
\begin{equation*}
\left\{\omega\in\Omega\relmiddle|\sup_{\bs{v}\in\R_\le^m}\left|\wh{h}_N(\bs{v})-h(\bs{v})\right|>\epsilon\right\}
\end{equation*}
is contained in
\begin{align*}
&\bigcup_{j=1}^2 \left\{\omega\in\Omega\relmiddle|\left|\wh{\pi}_{j,N}-\pi_j\right|>\frac{\epsilon}{4(m+2)}\right\}\\
&\ \cup \bigcup_{j=1}^2\left\{\omega\in\Omega\relmiddle|\sup_{x\in\R}\left|\wh{F}_{j,N}(x)-F_j(x)\right|>\frac{\epsilon}{8(m+1)}\right\},
\end{align*}
and therefore
\begin{align*}
&\sum_{N=1}^\infty\mathbb{P}\left(\left\{\omega\in\Omega\relmiddle|\sup_{\bs{v}\in\R_\le^m}\left|\wh{h}_N(\bs{v})-h(\bs{v})\right|>\epsilon\right\}\right)\\
&\le \sum_{j=1}^2\sum_{N=1}^\infty\mathbb{P}\left(\left\{\omega\in\Omega\relmiddle|\left|\wh{\pi}_{j,N}-\pi_j\right|>\frac{\epsilon}{4(m+2)}\right\}\right)\\ 
&\qquad + \sum_{j=1}^2\sum_{N=1}^\infty\mathbb{P}\left(\left\{\omega\in\Omega\relmiddle|\sup_{x\in\R}\left|\wh{F}_{j,N}(x)-F_j(x)\right|>\frac{\epsilon}{8(m+1)}\right\}\right)\\
&<\infty
\end{align*}
by \Cref{Fhconv,pihconv}.
\end{proof}

\begin{definition}
\upshape
Let $(A, d)$ be a metric space.
We define a discrepancy of $A_1\subset A$ from $A_2\subset A$ by
\begin{equation*}
D(A_1, A_2)=\sup_{a_1\in A_1}\left\{\inf_{a_2\in A_2}d(a_1,a_2)\right\}.
\end{equation*}
If $d$ is a Euclidean metric, we may write $D_\mathrm{E}$ in place of $D$.
\end{definition}

\begin{lemma}\label{supnotK}
Let $(A, d)$ be a metric space.
Let $g$ and $g_i$ (i=1,2,\ldots) be real functions on $A$ such that $\max\,\{g(t)\mid t\in A\}$ and $\max\,\{g_i(t)\mid t\in A\}$ exist.
Put $T=\argmax_{t\in A}\,\{g(t)\}$ and $T_i=\argmax_{t\in A}\,\{g_i(t)\}$.
Suppose $g$ is continuous on $A$, $\sup_{t\in A}|g_i(t)-g(t)|\to 0$ as $i\to\infty$, and there exists a compact set $K\subset A$ such that
\begin{equation*}
\sup\left\{g(t)\mid t\in A\setminus K\right\}<\max\left\{g(t)\mid t\in A\right\}.
\end{equation*}
Then $D(T_i,T)\to 0$ as $i\to\infty$.
\end{lemma}
\begin{proof}
Put $w_1=\max\,\{g(t)\mid t\in A\}$, $w_0=\sup\,\{g(t)\mid t\in A\setminus K\}$, and $w=(w_1-w_0)/3$.
(Note that $w_0<w_0+w<w_0+2w=w_1-w<w_1$.)
For any $\epsilon>0$, there exists $\delta>0$ such that $\delta<\epsilon$ and $g(t)>w_1-w$ for all $t\in T_\delta = \cup_{t\in T}\,\{x\in K\mid d(x, t)<\delta\}$, since $g$ is uniformly continuous on $K$ (see \cite[Theorem 4.19]{rudin76}).
(Note that $T\subset K$.)
Put $w_0'=\max\,\{g(t)\mid t\in K\setminus T_\delta\}$ (this exists because $K\setminus T_\delta$ is compact) and $w'>0$ such that $w'<w$ and $w_0'<w_0'+2w'<w_1$.
(Note that $\{t\in A\mid g(t)>w_1-2w'\}\subset T_\delta$ since $w_1-2w'>w_1-2w>w_0$ and $w_1-2w'>w_0'$.)
Since $\sup_{t\in A}|g_i(t)-g(t)|\to 0$ as $i\to\infty$, there is an integer $M$ such that $i\ge M$ implies $\sup_{t\in A}|g_i(t)-g(t)|<w'$.
Hence, for any $i\ge M$ and for all $t_1\in T_i$, we have  $g(t_1)>w_1-2w'$ (because $g(t_1)+w'>g_i(t_1)\ge g_i(t_2)>w_1-w'$ where $t_2\in T$), and thus $t_1\in T_\delta$.
Therefore, $\sup_{t_1\in T_i}\,\{\inf_{t_2\in T}d(t_1, t_2)\}\le \delta<\epsilon$ for any $i\ge M$.
Since $\epsilon$ was arbitrary, the claim follows.
\end{proof}

\begin{lemma}\label{Kexist}
There exists a compact set $K\subset\R_\le^n$ such that
\begin{equation*}
\sup\left\{h(\bs{v})\mid \bs{v}\in\R_\le^n\setminus K\right\}
< \max\left\{h(\bs{v})\mid \bs{v}\in\R_\le^n\right\}.
\end{equation*}
\end{lemma}
\begin{proof}
By \Cref{delISc,delISc2,hwhv2}, there exist 
\begin{equation*}
M_m=\max\left\{h(\bs{v})\mid \bs{v}\in\R_\le^m\right\}\qquad (m=1\ldots,n)
\end{equation*}
and $M=\max\,\{M_1,\ldots,M_{n-1}\}<M_n$.
Take $\epsilon>0$ such that $\epsilon<(M_n-M)/3$.
We can take $\alpha,\beta\in\R$ such that $F_j(\alpha)<\epsilon$ and $1-F_j(\beta)<\epsilon$ ($j=1,2$), since $F_j$ are non-decreasing functions with $\lim_{x\to -\infty}F_j(x)=0$ and $\lim_{x\to \infty}F_j(x)=1$.
Let $K=[\alpha, \beta]^n\cap\R_\le^n$ and $\bs{v}=(v_1,\ldots,v_n)\in\R_\le^n\setminus K$.
Then $v_1<\alpha$ or $v_n>\beta$ holds.

Suppose $v_1<\alpha$.
Put $\bs{v}'=(v_2,\dots,v_n)$ and recall that $\varphi_j(v,v')=\pi_j[F_j(v)-F_j(v')]$.
Using \Cref{xyzw}, we obtain
\begin{align*}
\left|h(\bs{v})-h(\bs{v}') \right|
&= \left|\max_{j}\left\{\varphi_j(v_1,-\infty)\right\} + \max_{j}\left\{\varphi_j(v_2,v_1)\right\} - \max_{j}\left\{\varphi_j(v_2,-\infty)\right\} \right| \\
&\le \left|\max_{j}\left\{\varphi_j(v_1,-\infty)\right\}\right| + \left|\max_{j}\left\{\varphi_j(v_2,v_1)\right\} - \max_{j}\left\{\varphi_j(v_2,-\infty)\right\} \right| \\
&<\epsilon + \left|\varphi_1(v_2,v_1) - \varphi_1(v_2,-\infty) \right| + \left|\varphi_2(v_2,v_1) - \varphi_2(v_2,-\infty) \right| \\
&=\epsilon + \left|\varphi_1(v_1, -\infty) \right| + \left|\varphi_2(v_1, -\infty) \right| \\
&< 3\epsilon.
\end{align*}
Hence $|h(\bs{v})|\le |h(\bs{v})-h(\bs{v}')| + |h(\bs{v}')| < 3\epsilon + M$.
We can similarly prove that $|h(\bs{v})|<3\epsilon + M$ for the case $v_n>\beta$.
Therefore, $\sup\,\{h(\bs{v})\mid \bs{v}\in\R_\le^n\setminus K\}\le 3\epsilon + M < (M_n-M) + M = M_n$.
This completes the proof.
\end{proof}

\begin{theorem}\label{DVn}
The discrepancy $D_\mathrm{E}(\wh{\V}_{n,N}, \V_n)$ converges completely to $0$ as $N\to\infty$.
\end{theorem}
\begin{proof}
In \Cref{supnotK}, let $(A, d)$ be the subspace $\R_\le^n$ of the Euclidean metric space $\R^n$, $g=h$ (which is continuous on $\R_\le^n$), and $g_i=\wh{h}_i$.
It follows from \Cref{VmCm,Kexist} that for any $\epsilon>0$, we can take $w'>0$ as in the proof of \Cref{supnotK}, and observe that $D_\mathrm{E}(\wh{\V}_{n,N}, \V_n)<\epsilon$ if $\sup_{\bs{v}\in\R_\le^n}|\wh{h}_N(\bs{v})-h(\bs{v})|<w'$.
(Note that $\max\,\{\wh{h}_N(\bs{v})\mid \bs{v}\in\wh{\R}_N^n\}=\max\,\{\wh{h}_N(\bs{v})\mid \bs{v}\in\R_\le^n\}$, hence $\wh{\V}_{n,N}\subset \argmax_{\bs{v}\in\R_\le^n}\{\wh{h}_N(\bs{v})\}$.)
This means that 
\begin{equation*}
\left\{\omega\in\Omega\relmiddle| \sup_{\bs{v}\in\R_\le^n}\left|\wh{h}_N(\bs{v})-h(\bs{v})\right|<w'\right\}\subset
\left\{\omega\in\Omega\relmiddle| D_\mathrm{E}\left(\wh{\V}_{n,N}, \V_n\right)<\epsilon\right\},
\end{equation*}
hence
\begin{equation*}
\left\{\omega\in\Omega\relmiddle| D_\mathrm{E}\left(\wh{\V}_{n,N}, \V_n\right)>\epsilon\right\}
\subset
\left\{\omega\in\Omega\relmiddle| \sup_{\bs{v}\in\R_\le^n}\left|\wh{h}_N(\bs{v})-h(\bs{v})\right|>\frac{w'}{2}\right\},
\end{equation*}
and therefore
\begin{align*}
&\sum_{N=1}^\infty\mathbb{P}\left(\left\{\omega\in\Omega\relmiddle| D_\mathrm{E}\left(\wh{\V}_{n,N}, \V_n\right)>\epsilon\right\}\right)\\
&\le \sum_{N=1}^\infty\mathbb{P}\left(\left\{\omega\in\Omega\relmiddle| \sup_{\bs{v}\in\R_\le^n}\left|\wh{h}_N(\bs{v})-h(\bs{v})\right|>\frac{w'}{2}\right\}\right)\\
&<\infty
\end{align*}
by \Cref{suphconv}.
\end{proof}

\begin{corollary}\label{v_conv_app}
The estimator $\wh{\bs{v}}_N\in\wh{\V}_{n,N}$ converges completely to $\bs{c}$ as $N\to\infty$.
\end{corollary}
\begin{proof}
Since $\V_n=\{\bs{c}\}$ by \Cref{delISc}, we have  $D_\mathrm{E}(\wh{\V}_{n,N}, \V_n) = \sup_{\bs{v}\in\wh{\V}_{n,N}}\|\bs{v}-\bs{c}\| \ge \|\wh{\bs{v}}_N-\bs{c}\|$.
Hence the claim follows from \Cref{DVn}.
\end{proof}

\begin{theorem}\label{rho_prf}
The estimator $\wh{\rho}_{\wh{\bs{v}}_N,N}$ converges completely to $\rho(\pi_1f_1,\pi_2f_2)$ as $N\to\infty$.
\end{theorem}
\begin{proof}
From \Cref{rhoeq,rhoheq}, we have
\begin{align}
&\rho(\pi_1f_1, \pi_2f_2) = \sum_{k=1}^{n+1} \min_j\left\{\pi_j\left[F_j(c_k)-F_j(c_{k-1})\right]\right\}, \\
&\wh{\rho}_{\wh{\bs{v}}_N,N} = \sum_{k=1}^{n+1}\min_j\left\{\wh{\pi}_{j,N}\left[\wh{F}_{j,N}(\wh{v}_k)-\wh{F}_{j,N}(\wh{v}_{k-1})\right]\right\},
\end{align}
where $\wh{\bs{v}}_N=(\wh{v}_1,\ldots,\wh{v}_n)\in\wh{\V}_{n,N}$, $\wh{v}_0=-\infty$, and $\wh{v}_{n+1}=\infty$.
By \Cref{xyzw},
\begin{align*}
&\left|\wh{\rho}_{\wh{\bs{v}}_N,N} - \rho(\pi_1f_1,\pi_2f_2) \right|\\
&\le \sum_{k=1}^{n+1}\sum_{j=1}^2\left|\wh{\pi}_{j,N}\left[\wh{F}_{j,N}(\wh{v}_k)-\wh{F}_{j,N}(\wh{v}_{k-1})\right] - \pi_j\left[F_j(c_k)-F_j(c_{k-1})\right] \right|,
\end{align*}
where 
\begin{align*}
&\left|\wh{\pi}_{j,N}\left[\wh{F}_{j,N}(\wh{v}_k)-\wh{F}_{j,N}(\wh{v}_{k-1})\right] - \pi_j\left[F_j(c_k)-F_j(c_{k-1})\right] \right| \\
&\le \left|\wh{\pi}_{j,N}\left[\wh{F}_{j,N}(\wh{v}_k)-\wh{F}_{j,N}(\wh{v}_{k-1})\right] - \pi_j\left[F_j(\wh{v}_k)-F_j(\wh{v}_{k-1})\right] \right| \\
&\qquad + \left|\pi_j\left[F_j(\wh{v}_k)-F_j(\wh{v}_{k-1})\right] - \pi_j\left[F_j(c_k)-F_j(c_{k-1})\right] \right| \\
&\le \left|\wh{\pi}_{j,N}\wh{F}_{j,N}(\wh{v}_k) - \pi_jF_j(\wh{v}_k)\right| + \left|\wh{\pi}_{j,N}\wh{F}_{j,N}(\wh{v}_{k-1}) - \pi_jF_j(\wh{v}_{k-1})\right| \\
&\qquad + \pi_j\left|F_j(\wh{v}_k)-F_j(c_k)\right| + \pi_j\left|F_j(\wh{v}_{k-1})-F_j(c_{k-1}) \right| \\
&\le \left|\wh{\pi}_{j,N}\wh{F}_{j,N}(\wh{v}_k) - \pi_j\wh{F}_{j,N}(\wh{v}_k)\right| + \left|\pi_j\wh{F}_{j,N}(\wh{v}_k) - \pi_jF_j(\wh{v}_k)\right| \\
&\qquad + \left|\wh{\pi}_{j,N}\wh{F}_{j,N}(\wh{v}_{k-1}) - \pi_j\wh{F}_{j,N}(\wh{v}_{k-1})\right| + \left|\pi_j\wh{F}_{j,N}(\wh{v}_{k-1}) - \pi_jF_j(\wh{v}_{k-1})\right|\\
&\qquad + \pi_j\left|F_j(\wh{v}_k)-F_j(c_k)\right| + \pi_j\left|F_j(\wh{v}_{k-1})-F_j(c_{k-1}) \right| \\
&\le 2\left|\wh{\pi}_{j,N}-\pi_j \right| + \pi_j\left|\wh{F}_{j,N}(\wh{v}_k) - F_j(\wh{v}_k)\right| + \pi_j\left|\wh{F}_{j,N}(\wh{v}_{k-1}) - F_j(\wh{v}_{k-1})\right| \\
&\qquad + \pi_j\left|F_j(\wh{v}_k)-F_j(c_k)\right| + \pi_j\left|F_j(\wh{v}_{k-1})-F_j(c_{k-1}) \right|.
\end{align*}
Hence
\begin{equation}\label{rhohineq}
\begin{aligned}
&\left|\wh{\rho}_{\wh{\bs{v}}_N,N} - \rho(\pi_1f_1,\pi_2f_2) \right| \\
&\le 2(n+1)\sum_{j=1}^2\left|\wh{\pi}_{j,N}-\pi_j \right|\\
&\qquad + \sum_{k=1}^{n+1}\sum_{j=1}^2\pi_j\left|\wh{F}_{j,N}(\wh{v}_k) - F_j(\wh{v}_k)\right| + \sum_{k=1}^{n+1}\sum_{j=1}^2\pi_j\left|\wh{F}_{j,N}(\wh{v}_{k-1}) - F_j(\wh{v}_{k-1})\right| \\
&\qquad + \sum_{k=1}^{n+1}\sum_{j=1}^2\pi_j\left|F_j(\wh{v}_k)-F_j(c_k)\right| + \sum_{k=1}^{n+1}\sum_{j=1}^2\pi_j\left|F_j(\wh{v}_{k-1})-F_j(c_{k-1}) \right|\\
&= 2(n+1)\sum_{j=1}^2\left|\wh{\pi}_{j,N}-\pi_j \right|\\
&\qquad + 2\sum_{k=1}^{n}\sum_{j=1}^2\pi_j\left|\wh{F}_{j,N}(\wh{v}_k) - F_j(\wh{v}_k)\right| + 2\sum_{k=1}^{n}\sum_{j=1}^2\pi_j\left|F_j(\wh{v}_k)-F_j(c_k)\right|.
\end{aligned}
\end{equation}
For any $\epsilon>0$, there exists $\delta>0$ such that $|F_j(x)-F_j(c_k)|<\epsilon/(6n)$ for all $x\in\R$ with $|x-c_k|<\delta$ ($j=1,2;\ k=1,\ldots,n$).
If
\begin{equation*}
\left|\wh{\pi}_{j,N}-\pi_j \right| < \frac{\epsilon}{12(n+1)},\qquad
\sup_{x\in\R}\left|\wh{F}_{j,N}(x)-F_j(x)\right| < \frac{\epsilon}{6n},\qquad
\left|\wh{v}_k-c_k \right|<\delta,
\end{equation*}
for $j=1,2$ and $k=1,\ldots,n$, then
\begin{equation*}
\left|\wh{\rho}_{\wh{\bs{v}}_N,N} - \rho(\pi_1f_1,\pi_2f_2) \right|
<2(n+1)\frac{2\epsilon}{12(n+1)}+2n(\pi_1+\pi_2)\frac{\epsilon}{6n} + 2n(\pi_1+\pi_2)\frac{\epsilon}{6n} = \epsilon
\end{equation*}
by \Cref{rhohineq}.
Hence $\{\omega\in\Omega\mid |\wh{\rho}_{\wh{\bs{v}}_N,N}-\rho(\pi_1f_1,\pi_2f_2)|>\epsilon\}$ is contained in 
\begin{align*}
&\bigcup_{j=1}^2 \left\{\omega\in\Omega\relmiddle|\left|\wh{\pi}_{j,N}-\pi_j\right|>\frac{\epsilon}{24(n+1)}\right\}\\
&\ \cup \bigcup_{j=1}^2 \left\{\omega\in\Omega\relmiddle|\sup_{x\in\R}\left|\wh{F}_{j,N}(x)-F_j(x)\right|>\frac{\epsilon}{12n}\right\}\\
&\ \cup \left\{\omega\in\Omega\relmiddle| \left\|\wh{\bs{v}}_N-\bs{c}\right\|>\frac{\delta}{2}\right\},
\end{align*}
and therefore
\begin{align*}
&\sum_{N=1}^\infty\mathbb{P}\left(\left\{\omega\in\Omega\relmiddle| \left|\wh{\rho}_{\wh{\bs{v}}_N,N}-\rho(\pi_1f_1,\pi_2f_2)\right|>\epsilon\right\}\right)\\
&\le \sum_{j=1}^2\sum_{N=1}^\infty\mathbb{P}\left(\left\{\omega\in\Omega\relmiddle|\left|\wh{\pi}_{j,N}-\pi_j\right|>\frac{\epsilon}{24(n+1)}\right\}\right)\\ 
&\qquad + \sum_{j=1}^2\sum_{N=1}^\infty\mathbb{P}\left(\left\{\omega\in\Omega\relmiddle|\sup_{x\in\R}\left|\wh{F}_{j,N}(x)-F_j(x)\right|>\frac{\epsilon}{12n}\right\}\right)\\
&\qquad + \sum_{N=1}^\infty\mathbb{P}\left(\left\{\omega\in\Omega\relmiddle| \left\|\wh{\bs{v}}_N-\bs{c}\right\|>\frac{\delta}{2}\right\}\right) \\
&<\infty
\end{align*}
by \Cref{Fhconv}, \Cref{pihconv}, and \Cref{v_conv_app}.
\end{proof}

Note that \Cref{v_conv_app,rho_prf} are exactly \Cref{v_conv,rho_conv}, respectively.

As stated above, we have estimated $\bs{c}$ as $\wh{\bs{v}}_N\in\wh{\V}_{n,N}$.
In fact, it is possible to estimate $\bs{c}$ in another way.
For ${\bs v}=(v_1,\ldots,v_m)\in \R_\le^m$ with $m$ a positive integer, let us define
\begin{equation}\label{rhoeq3}
\rho_{\bs v} = \sum_{k=1}^{m+1}\min_j\left\{\varphi_j(v_k,v_{k-1})\right\},
\end{equation}
where $v_0=-\infty$ and $v_{m+1}=\infty$.
Note that we have 
\begin{equation}\label{rhoheq3}
\wh{\rho}_{\bs{v},N} = \sum_{k=1}^{m+1}\min_j\left\{\wh{\varphi}_j(v_k,v_{k-1})\right\}
\end{equation}
by \Cref{rhoheq}.
Here recall that
\begin{align*}
&\varphi_j(v_k,v_{k-1})=\pi_j[F_j(v_k)-F_j(v_{k-1})],\\
&\wh{\varphi}_{j,N}(v_k,v_{k-1})=\wh{\pi}_{j,N}[\wh{F}_{j,N}(v_k)-\wh{F}_{j,N}(v_{k-1})].
\end{align*}

\begin{lemma}\label{hplusrho}
For $\bs{v}=(v_1,\ldots,v_m)\in\R_\le^m$ with $m$ a positive integer, we have
\begin{align}
&h(\bs{v})+\rho_{\bs v} = 1-\max_j\left\{\pi_j\right\},\label{eq:hrho1}\\
&\wh{h}_N(\bs{v}) + \wh{\rho}_{\bs{v},N} = 1-\max_j\left\{\wh{\pi}_{j,N}\right\}.\label{eq:hrho2}
\end{align}
\end{lemma}
\begin{proof}
For $k=1,\ldots,m+1$, choose 
\begin{align*}
&j_k\in\argmax_j\left\{\varphi_j(v_k,v_{k-1})\right\},\\ &l_k\in\argmin_j\left\{\varphi_j(v_k,v_{k-1})\right\}
\end{align*}
such that $\{j_k,l_k\}=\{1,2\}$, where $v_0=-\infty$ and $v_{m+1}=\infty$.
By \Cref{hveq,rhoeq3}, we have
\begin{align*}
h(\bs{v})+\rho_{\bs v}
&= \sum_{k=1}^{m+1} \left[\max_j\left\{\varphi_j(v_k,v_{k-1})\right\} + \min_j\left\{\varphi_j(v_k,v_{k-1})\right\}\right] - \max_j\left\{\pi_j\right\}\\
&= \sum_{k=1}^{m+1} \left[\varphi_{j_k}(v_k,v_{k-1}) + \varphi_{l_k}(v_k,v_{k-1})\right] - \max_j\left\{\pi_j\right\}\\
&= \sum_{k=1}^{m+1} \left[\varphi_1(v_k,v_{k-1}) + \varphi_2(v_k,v_{k-1})\right] - \max_j\left\{\pi_j\right\}\\
&= \sum_{j=1}^2\sum_{k=1}^{m+1}\pi_j\left[F_j(v_k)-F_j(v_{k-1})\right] - \max_j\left\{\pi_j\right\}\\
&= \sum_{j=1}^2\pi_j\left[F_j(\infty) - F_j(-\infty)\right] - \max_j\left\{\pi_j\right\}\\
&= 1 - \max_j\left\{\pi_j\right\},
\end{align*}
which implies \Cref{eq:hrho1}.

We can prove \Cref{eq:hrho2} in a similar way.
For $k=1,\ldots,m+1$, redefine 
\begin{align*}
&j_k\in\argmax_j\left\{\wh{\varphi}_j(v_k,v_{k-1})\right\},\\ &l_k\in\argmin_j\left\{\wh{\varphi}_j(v_k,v_{k-1})\right\}
\end{align*}
such that $\{j_k,l_k\}=\{1,2\}$.
By \Cref{hhveq,rhoheq3}, we have
\begin{align*}
\wh{h}_N(\bs{v}) + \wh{\rho}_{\bs{v},N}
&= \sum_{k=1}^{m+1} \left[\max_j\left\{\wh{\varphi}_{j,N}(v_k,v_{k-1})\right\} + \min_j\left\{\wh{\varphi}_{j,N}(v_k,v_{k-1})\right\}\right] - \max_j\left\{\wh{\pi}_{j,N}\right\}\\
&= \sum_{k=1}^{m+1} \left[\wh{\varphi}_{j_k,N}(v_k,v_{k-1}) + \wh{\varphi}_{l_k,N}(v_k,v_{k-1})\right] - \max_j\left\{\wh{\pi}_{j,N}\right\}\\
&= \sum_{k=1}^{m+1} \left[\wh{\varphi}_{1,N}(v_k,v_{k-1}) + \wh{\varphi}_{2,N}(v_k,v_{k-1})\right] - \max_j\left\{\wh{\pi}_{j,N}\right\}\\
&= \sum_{j=1}^2\sum_{k=1}^{m+1}\wh{\pi}_{j,N}\left[\wh{F}_{j,N}(v_k)-\wh{F}_{j,N}(v_{k-1})\right] - \max_j\left\{\wh{\pi}_{j,N}\right\}\\
&= \sum_{j=1}^2\wh{\pi}_{j,N}\left[\wh{F}_{j,N}(\infty) - \wh{F}_{j,N}(-\infty)\right] - \max_j\left\{\wh{\pi}_{j,N}\right\}\\
&= 1 - \max_j\left\{\wh{\pi}_{j,N}\right\},
\end{align*}
which implies \Cref{eq:hrho2}.
\end{proof}

It is immediate from \Cref{hplusrho} that 
\begin{align}
&\argmin_{\bs{v}\in\R_\le^m} \left\{\rho_{\bs v}\right\} = \V_m,\label{argminrho} \\
&\argmin_{\bs{v}\in\wh{\R}_N^m} \left\{\wh{\rho}_{\bs{v},N}\right\} = \wh{\V}_{m,N}\label{argminrhoh}
\end{align}
for $m=1,\ldots,n$.

\begin{theorem}
For $\bs{v}\in\R_\le^n$, $\rho_{\bs v}$ attains its unique minimum $\rho(\pi_1f_1,\pi_2f_2)$ at $\bs{v}=\bs{c}$.
\end{theorem}
\begin{proof}
This follows from \Cref{delISc}, \Cref{rhoeq}, and \Cref{argminrho}.
\end{proof}

\begin{theorem}
Let $\wh{\bs{v}}_N'\in \argmin_{\bs{v}\in\wh{\R}_N^n} \{\wh{\rho}_{\bs{v},N}\}$.
Then $\wh{\bs{v}}_N'$ converges completely to $\bs{c}$ as $N\to\infty$.
Furthermore, $\wh{\rho}_{\wh{\bs{v}}_N',N}$ converges completely to $\rho(\pi_1f_1,\pi_2f_2)$ as $N\to\infty$.
\end{theorem}
\begin{proof}
Since $\wh{\bs{v}}_N'\in\wh{\V}_{n,N}$ by \Cref{argminrhoh}, the claim follows from \Cref{v_conv_app,rho_prf}.
\end{proof}

\section{Measurability of some functions}\label{measurability}
\subsection{The measurability of $\unboldmath\wh{\rho}_{\wh{\bs{v}}_N,N}$ (associated with \cref{rho_conv,rho_prf})}
It follows from \cref{argminrhoh} that $\wh{\rho}_{\wh{\bs{v}}_N,N} = \min_{\bs{v}\in\wh{\R}_N^n}\{\wh{\rho}_{\bs{v},N}\}$, which depends only on the rank statistics of $X_1,...,X_N$ (labeled by $Y_1,\ldots,Y_N$, respectively).
We then see that $\{\wh{\rho}_{\wh{\bs{v}}_N,N}(\omega)\mid \omega\in\Omega\}$ is a finite set and that $\wh{\rho}_{\wh{\bs{v}}_N,N}$ is a measurable simple function on $\Omega$.

\subsection{The measurability of $\unboldmath\sup_{\bs{v}\in\R_\le^m}|\wh{h}_N(\bs{v})-h(\bs{v})|$ (associated with \cref{suphconv})}
By the right continuity of $\wh{F}_{j,N}$ (\cref{def:A1}), we see that $\sup_{\bs{v}\in\R_\le^m}|\wh{h}_N(\bs{v})-h(\bs{v})| = \sup_{\bs{v}\in\mathbb{Q}_\le^m}|\wh{h}_N(\bs{v})-h(\bs{v})|$ for any positive integer $m$, where $\mathbb{Q}$ is the set of rational numbers and $\mathbb{Q}_\le^m = \{(v_1,\ldots,v_m)\in\mathbb{Q}^m\mid v_1\le\cdots\le v_m\}$.
Since $\mathbb{Q}_\le^m$ is countable and $\wh{h}_N(\bs{v})$ is obviously measurable on $\Omega$, $\sup_{\bs{v}\in\R_\le^m}|\wh{h}_N(\bs{v})-h(\bs{v})|$ is also measurable on $\Omega$.

\subsection{The measurability of $\unboldmath D_\mathrm{E}(\wh{\V}_{n,N}, \V_n)$ (associated with \cref{DVn})}
Let $\{K_1,\ldots,K_m\}$ be the collection of all nonempty subsets of $\{(i_1,\ldots,i_n)\mid 1\le i_1\le\cdots\le i_n\le N-1\}$ with $K_j\ne K_l$ if ($j\ne l$),  $\Omega_{K_j}$ be the set of all $\omega\in\Omega$ such that $\wh{\V}_{n,N} = \{(Z_{i_1}, \ldots, Z_{i_n})\mid (i_1,\ldots,i_n)\in K_j\}$.
Then $\Omega_{K_j}\in\mathcal{F}$ for all $j\in\{1,\ldots,m\}$,
$\Omega = \cup_{j=1}^m \Omega_{K_j}$, and $\Omega_{K_j}\cap\Omega_{K_l}=\emptyset$ if $j\ne l$.
Since the restriction of $D_\mathrm{E}(\wh{\V}_{n,N}, \V_n)$ to each $\Omega_{K_j}$ coincides with $\max\,\{\|(Z_{i_1}, \ldots, Z_{i_n})-\bs{c}\|\mid (i_1,\ldots,i_n)\in K_j\}$, which is measurable on $\Omega_{K_j}$, we see that $D_\mathrm{E}(\wh{\V}_{n,N}, \V_n)$ is measurable on $\Omega$.

\subsection{The measurability of $\unboldmath \wh{\bs{v}}_N$ (associated with \cref{v_conv,v_conv_app})}
We can choose $\wh{\bs{v}}_N\in\wh{\V}_{n,N}$ such that $\wh{\bs{v}}_N:\Omega\to\R^n$ is measurable.
Indeed, let $\mathcal{I}=\{(i_1,\ldots,i_n)\mid 1\le i_1\le\cdots\le i_n\le N-1\}$ and $\mathbb{Z}_{(i_1,\ldots,i_n)}=(Z_{i_1},\ldots,Z_{i_n})\in\widehat{\R}_N^n$ for $(i_1,\ldots,i_n)\in\mathcal{I}$.
Note that $\Omega$ equals the disjoint union of measurable sets
\[
\Omega_{\mathcal{J}} = \left\{\omega\in\Omega\mid \wh{h}_N(\mathbb{Z}_{\bs j}) = \max_{\bs{i}\in\mathcal{I}}\wh{h}_N(\mathbb{Z}_{\bs i})\text{ if and only if }\bs{j}\in\mathcal{J}\right\}
\]
over all nonempty subsets $\mathcal{J}$ of $\mathcal{I}$.
For such $\mathcal J$, we can define $\max \mathcal{J}$ and $\min \mathcal{J}$ in lexicographic order.
If we put $\wh{\bs{v}}_N = \mathbb{Z}_{\max \mathcal{J}}$ (or $\wh{\bs{v}}_N = \mathbb{Z}_{\min \mathcal{J}}$) on each $\Omega_{\mathcal{J}}$, then $\wh{\bs{v}}_N$ is measurable.

If we choose $\wh{\bs{v}}_N\in\wh{\V}_{n,N}$ at random independently of $(\Omega,\mathcal{F},\mathbb{P})$, we cannot guarantee that $\wh{\bs{v}}_N$ is measurable.
In such a case,
we mean by ``$\wh{\bs{v}}_N$ converges completely to $\bs{c}$ as $N\to\infty$'' that for any $\epsilon>0$, there exists a collection $\{A_1,A_2,\ldots\}$ of measurable sets such that $\sum_{N=1}^\infty \mathbb{P}(A_N)<\infty$ and $A_N\supset \{\omega\in\Omega\mid \|\wh{\bs{v}}_N-\bs{c}\|>\epsilon\}$ for all $N$, which also implies that $\wh{\bs{v}}_N$ converges almost surely to $\bs{c}$ (in the sense that $\mathbb{P}(\{\omega\in\Omega\mid\lim_{N\to\infty} \wh{\bs{v}}_N = \bs{c}\}) =1$) if $(\Omega,\mathcal{F},\mathbb{P})$ is complete (see \cref{rem:compas}).
In fact, we can take $A_N = \{\omega\in\Omega\mid D_\mathrm{E}(\wh{\V}_{n,N}, \V_n)>\epsilon\}$.

\section{Additional proofs}\label{AppB}
In this section, \Cref{first_C,first_rho,second_C,second_rho} will be proved.
We shall take over the notations in \Cref{sec:exper}.

\begin{proposition}
In the first case, 
\begin{align*}
&C(\pi_1f_1, \pi_2f_2) = \{c_1\} = \left\{(\log 2)/2\right\},\\
&\rho(\pi_1f_1, \pi_2f_2) = \left[2-2\Phi\left(c_1+1\right) + \Phi\left(c_1-1\right)\right]/3.
\end{align*}
\end{proposition}
\begin{proof}
The equation $\pi_1f_1(x) = \pi_2f_2(x)$ gives $x=(\log 2)/2$, which is a crossover point.
Hence $C(\pi_1f_1, \pi_2f_2) = \{c_1\} = \left\{(\log 2)/2\right\}$.
Next, 
\begin{align*}
\rho(\pi_1f_1, \pi_2f_2)
&= \pi_2F_2(c_1)+\pi_1[1-F_1(c_1)] \\
&= \frac{1}{3}\Phi(c_1-1) + \frac{2}{3}\left[1-\Phi(c_1+1)\right].
\end{align*}
\end{proof}

\begin{proposition}
In the second case, 
\begin{align*}
&C(\pi_1f_1, \pi_2f_2) = \{c_1, c_2\} = \cosh^{-1}\left(0.8\sqrt{\mathrm{e}}\right), \\
&\rho(\pi_1f_1, \pi_2f_2) = 0.8 - 0.5\Phi(c_1+1) + 0.5\Phi(c_2+1) - 0.8\Phi(c_2).
\end{align*}
\end{proposition}
\begin{proof}
If $x<0$ or $x>0.5$, then $f_2(x)=0.8\nu_{0,1}(x)$, and $\pi_1f_1(x) = \pi_2f_2(x)$ gives $\cosh (x)=0.8\sqrt{\mathrm{e}}$.
There is a unique $c>0$ such that $\cosh (c)=0.8\sqrt{\mathrm{e}}$.
Since $c>0.5$ and $\pi_1f_1<\pi_2f_2$ on $[0, 0.5]$, we have $C(\pi_1f_1, \pi_2f_2) = \{-c, c\} = \cosh^{-1}(0.8\sqrt{\mathrm{e}})$.
Next, 
\begin{align*}
\rho(\pi_1f_1, \pi_2f_2)
&= \pi_2F_2(-c)+\pi_1[F_1(c)-F_1(-c)]+\pi_2[1-F_2(c)] \\
&= 0.8 - 0.5\Phi(-c+1) + 0.5\Phi(c+1) - 0.8\Phi(c).
\end{align*}
\end{proof}
\end{appendix}

 \section*{Acknowledgements}
This study was partially supported by
JSPS KAKENHI Grant Numbers JP15K04814, JP20K03509, and JP21K15762.
We thank Atsushi Komaba (University of Yamanashi) for the validation of our numerical results.



\bibliographystyle{imsart-number} 
\bibliography{references}       


\end{document}